%% file: main.tex
\documentclass[11pt]{article}

\usepackage[margin=1in]{geometry}
\usepackage{amsmath,amsthm,amssymb,amsfonts,float}
\usepackage{paralist}
\usepackage{verbatim}

\makeatletter
\newcommand{\subjclass}[2][2020]{%
  \let\@oldtitle\@title%
  \gdef\@title{\@oldtitle\footnotetext{#1 \emph{Mathematics subject classification.} #2}}%
}
\makeatother

\usepackage[colorlinks=true, pdfstartview=FitV, linkcolor=blue,citecolor=blue, urlcolor=blue]{hyperref}

\usepackage[abbrev,lite,nobysame]{amsrefs}
\usepackage{times}
\usepackage[usenames,dvipsnames]{color}
\usepackage{esint}
\usepackage{tikz-cd}
\usepackage{mathtools,enumitem,mathrsfs}

\usepackage[compact]{titlesec}

\usepackage[title]{appendix}

\usepackage{comment}

\input{macros.tex}

\numberwithin{equation}{section}
    
\begin{document}

\title{Singular Lagrangians in the Hitchin moduli space and conformal limits} 
\author{Sze-Hong Kwong\thanks{\footnotesize Department of Mathematics, University of Maryland, College Park, MD 20742, USA \href{mailto:@umd.edu}{\texttt{shkwong@umd.edu}}. }}%was supported by}}

\maketitle

\begin{abstract}
In the moduli space $\mathcal{M}_{H}$ of semistable $\text{SL}(r,\C)$-Higgs bundles, we show that there exists a sublocus $S$ of the upward flow through a polystable $\C^{*}$-fixed point $[(\dbar_0,\Phi_0)]$, which is Lagrangian on its intersection with the stable locus. When $\Phi_0=0$ or when the automorphism group of $(\dbar_0,\Phi_0)$ is abelian, we demonstrate that the intersection of $S$ with the stable locus is always non-empty, and establish the existence of conformal limits for stable Higgs bundles lying on $S$.
\end{abstract}

\tableofcontents

\setcounter{section}{0}

%\setcounter{tocdepth}{1}
%{\small\tableofcontents}

%% Introduction
\input{intro}

\input{prelim}
\input{model}

\input{proof}

%% Bibliography
\phantomsection
\addcontentsline{toc}{section}{References}
\bibliographystyle{abbrv}
\bibliography{bibliography}

\end{document}

%% file: macros.tex
\newcommand{\C}{\mathbb{C}}
\newcommand{\R}{\mathbb{R}}

\newcommand{\Q}{\mathbb{Q}}
\newcommand{\Z}{\mathbb{Z}}

\newcommand{\tr}{\text{tr}}

\newcommand{\sllie}{\mathfrak{sl}}

\newcommand{\dbar}{\overline{\partial}}

\newcommand{\git}{\mathbin{
  \mathchoice{/\mkern-6mu/}% \displaystyle
    {/\mkern-6mu/}% \textstyle
    {/\mkern-5mu/}% \scriptstyle
    {/\mkern-5mu/}}}% \scriptscriptstyle

\usepackage{bbm}

\newtheorem{theorem}{Theorem}[section]
\newtheorem{proposition}[theorem]{Proposition}
\newtheorem{corollary}[theorem]{Corollary}
\newtheorem{lemma}[theorem]{Lemma}
\newtheorem*{lemma*}{Lemma}

\newtheorem{assumption}{Assumption}

\theoremstyle{definition}

\newtheorem{remark}[theorem]{Remark}

%% file: intro.tex
\section{Introduction}
Over a fixed a closed Riemann surface of genus greater than one, a homeomorphism between the moduli space of degree zero polystable Higgs bundles and completely reducible flat connections, denoted by $\mathcal{M}_{0}$ and $\mathcal{M}_1$ respectively, is given by the celebrated nonabelian Hodge correspondence (NAH). In addition, identifying flat connections with holomorphic connections, the two singular complex spaces both admit a natural stratification coming from the Bia\l ynicki-Birula type decomposition on the moduli space of $\lambda$-connections $\mathcal{M}$, which is associated with the $\C^{*}$-action equivariant with the fibration $\pi: \mathcal{M}\to \C$, where one makes the identification $\mathcal{M}_{0}\cong \pi^{-1}(0)$ and $\mathcal{M}_{1}\cong \pi^{-1}(1)$ (see \cite{Sim10}). Restriction of this $\C^{*}$-action to $\mathcal{M}_0$ is the standard $\C^{*}$-action $t\cdot [(\mathcal{E},\Phi)]\mapsto [(\mathcal{E},t\Phi)]$. Each stratum in $\mathcal{M}$ is fibred over a connected component of the locus of $\C^{*}$-fixed points. Now, given a $\C^{*}$-fixed point $[(\mathcal{E}_0,\Phi_0)]$, let $W_0$ and $W_1$ be the restriction of the fibre to $\mathcal{M}_{0}$ and $\mathcal{M}_{1}$ respectively. The fibre $W_0$ is known as the \emph{upward flow} through $[(\mathcal{E}_0,\Phi_0)]$. It is an interesting object to study, especially when it is very stable, in which case $W_0$ is closed and plays a role in mirror symmetry (cf. \cite{HH22}).

When the Higgs pair $(\mathcal{E}_0,\Phi_0)$ is stable, much is understood about these restricted fibres $W_0$ and $W_1$. For instance, $W_{0}$ is contained in $\mathcal{M}^{s}_{0}$, the locus of stable Higgs bundles. The upward flow $W_0$ is locally closed complex Lagrangian and is $\C^{*}$-equivariantly isomorphic to the direct sum of positively graded subspaces $T^{+}_{[(\mathcal{E}_0,\Phi_0)]}\mathcal{M}_0\subset T_{[(\mathcal{E}_0,\Phi_0)]}\mathcal{M}_0$. In the analytic setting, one may replace $T^{+}_{[(\mathcal{E}_0,\Phi_0)]}\mathcal{M}_0$ with its harmonic representative $\mathcal{H}^{1}_{+}$. Through the Kuranishi map, a global parametrization $\mathcal{H}^{1}_{+}\to W_0$ is obtained in \cite{CW19}, which can be interpolated to give a parametrization to $W_1$ as well. Composing these two parametrizations then gives a correspondence from $W_0$ to $W_1$, which turns out to coincide with the mapping given by \emph{conformal limit} (see \cite{DFK+21} and \cite{CW19}). More detail about conformal limits will be given below.

The purpose of this article is to take the first step into the study of the singular cases, i.e. when $[(\mathcal{E}_0,\Phi_0)]$ is a singular point in $\mathcal{M}_0$, which happens when $(\mathcal{E}_0,\Phi_0)$ is strictly polystable. We will focus on the Dolbeault side. The main results will be proved under either one of the following two assumptions on the polystable representative $(\dbar_0,\Phi_0)$:

\begin{assumption}\label{assumption_I}
    $\Phi_0 =0$.
\end{assumption}

\begin{assumption}\label{assumption_II}
The automorphism group of $(\mathcal{E}_0,\Phi_0)$ is abelian.
\end{assumption}

Note that Assumption \ref{assumption_I} amounts to a saying that $\mathcal{E}_0$ is a polystable holomorphic vector bundle, while Assumption \ref{assumption_II} is equivalent to the statement that $(\mathcal{E}_0,\Phi_0)$ is a direct sum of pairwise non-isomorphic stable Higgs bundles of the degree zero. When the rank of the underlying vector bundle is at most three, at least one of the two assumptions are satisfied (cf. Proposition \ref{proposition_shb_block} below).

The structure of the upward flow through a singular point $[(\dbar_0,\Phi_0)]$ could be complicated, as opposed to the stable case, where the upward flow is parametrized by an affine space. However, when set up appropriately, the recipe for constructing this affine space parametrization will yield an interesting sublocus of the full upward flow.

\begin{theorem}\label{theorem_lagrangian}
Along the upward flow through $[(\dbar_0,\Phi_0)]$, there exists a sub-locus $$S=S([(\dbar_0,\Phi_0)])$$
parametrized by an affine GIT quotient, which is complex Lagrangian along the points which represent stable Higgs bundles. When Assumption \ref{assumption_I} or Assumption \ref{assumption_II} holds, $S$ contains some stable Higgs bundle.
\end{theorem}
We will call $S$ the \emph{central locus} of the upward flow through $[(\dbar_0,\Phi_0)]$. Its construction will be carried out in the next section. We remark that as an immediate consequence of Theorem \ref{theorem_lagrangian}, we have the following
\begin{corollary}
If Assumption \ref{assumption_I} or Assumption \ref{assumption_II} holds, then the upward flow through $[(\dbar_0,\Phi_0)]$ contains some stable Higgs bundle.
\end{corollary}

Another motivation for the study of the central locus $S$, or more generally the upward flow $W_0$ through a $\C^{*}$-fixed point comes from the notion of conformal limits, which are defined as follows. Let $(\dbar_E,\Phi)$ be a stable Higgs bundle. Then for each $R>0$, $(\dbar_{E},R\Phi)$ is also stable. Let $h_R$ be the harmonic metric for $(\dbar_{E},R\Phi)$, and $\dbar_{E}+\partial^{h_R}$ the corresponding Chern connection. For each $\xi\in \C-\{0\}$, we have the flat connection $$\dbar_{E}+\partial^{h_R}_{E}+\xi^{-1}R\Phi+\xi R\Phi^{*_{h_R}}.$$
Fixing a constant $\hbar \in \C-\{0\}$ and letting $R=\hbar^{-1}\xi$, the limit $$\lim_{R\to 0} (\dbar_{E}+\partial^{h_R}_{E}+\hbar^{-1}\Phi+R^{2}\Phi^{*_{h_R}}),$$
if exist, is said to be the \emph{($\hbar$-)conformal limit} of $(\dbar_{E},\Phi)$. When $(\dbar_{E},\Phi)$ lies on the Hitchin component, it was conjectured by Gaiotto in \cite{Gai14}, that its conformal limit exists and lies in the space of oper. The origin of this conjecture comes from superconformal field theory of class $\mathcal{S}$ which arise from compactification on the product of a Riemann surface and a circle of radius $R$. The non-punctured case was later proved in \cite{DFK+21}, whose methodology was then generalized in \cite{CW19} to show the existence of the conformal limit for $(\dbar_{E},\Phi)$ lying on the global parametrization of the upward flow through a stable $\C^{*}$-fixed point. The case for parabolic Higgs bundles is treated in \cite{CFW24}. One major obstacle absent from the stable case is the presence of non-trivial automorphism group, and hence the non-uniqueness of harmonic metrics for the $\C^{*}$-fixed point $(\dbar_0,\Phi_0)$, where $\lim_{t\to 0}[(\dbar_{E},t\Phi)]=[(\dbar_0,\Phi_0)]$. In the last section, we will prove the following

\begin{theorem}\label{theorem_CL}
Suppose either Assumption \ref{assumption_I} or Assumption \ref{assumption_II} holds for $(\dbar_0,\Phi_0)$. Let $(\dbar_{E},\Phi)$ be a stable Higgs bundle such that $[(\dbar_E,\Phi)]\in S([(\dbar_0,\Phi_0)])$. Then the conformal limit of $(\dbar_{E},\Phi)$ exists.
\end{theorem}

A few remarks are in order. For a strictly polystable $(\dbar_{E},\Phi)\in S$, one can decompose $(\dbar_{E},\Phi)$ into a direct sum of stable Higgs bundles and establish the existence of conformal limits of these stable factors. In addition, one may also consider a polystable Higgs pair $(\dbar_{E},\Phi)$ where $[(\dbar_{E},\Phi)]$ does not lie in the central locus $S$ through $[(\dbar_0,\Phi_0)]$. The reason we do not consider this case in this paper is that we have not constructed a convenient model for the full upward flow through $[(\dbar_0,\Phi_0)]$. Nonetheless, under Assumption \ref{assumption_II}, the method in the proof of Theorem \ref{theorem_CL} can be applied to such $(\dbar_{E},\Phi)$, provided that $(\dbar_{E},\Phi)$ is gauge transformed into a sufficiently `nice' form.

The rest of the paper is organized as follows. Section \ref{section_local_model} contains a collection of notations and preliminary results that we will need later. In Section \ref{section_slice_S_kuranishi_model}, we specialize to a polystable Higgs bundle $(\dbar_{E},\Phi)$ that represents a $\C^{*}$-fixed point in the moduli space. We then realize $(\dbar_{E},\Phi)$ as system of Hodge bundle, with which we define a grading on the deformation complex whose graded pieces are preserved by the automorphism group of $(\dbar_{E},\Phi)$. We then proceed to the construction of the model for the central locus and prove Theorem \ref{theorem_lagrangian}. The most technical part of the paper is the last subsection, where we prove Theorem \ref{theorem_CL}.

%% file: prelim.tex
\section{Kuranishi local model around a polystable Higgs bundle}\label{section_local_model}

The purpose of this section is to set up notation and collect preliminaries. We begin with an overview of the relevant aspects of the moduli space of semistable Higgs bundles from the gauge theoretic perspective, where we also introduce the notation adhered to throughout the paper. Next, we recall the local Kuranishi model around a polystable Higgs bundle constructed in \cite{Fan22}, which is a major ingredient needed in the discussion of the central locus and its model in the next section. We then review the Kempf-Ness criterion for stability in representation spaces under the action of a connected complex reductive Lie group. In the special case of complex tori, we record a criterion for stability of a vector in terms of the convex polytope associated with its effective weight.

\subsection{The moduli space of semistable Higgs bundles}

Throughout the paper, we fix a compact Riemann surface $X$ of genus $g\geq 2$, equipped with a K\"{a}hler form of unit volume. Denote the canonical bundle of $X$ by $K_X$ and the trivial holomorphic line bundle by $\mathcal{O}_{X}$. Let $E\to X$ be a smooth complex vector bundle of rank $r$ and degree zero, together with a fixed identification $\wedge^{r}E\xrightarrow{\cong} \underline{\C}$ with the trivial line bundle. By a holomorphic vector bundle with $E$ its underlying smooth vector bundle, we mean a pair $\mathcal{E}=(E,\dbar_{E})$ where $\dbar_{E}$ is a Dolbeault operator $\dbar_{E}:\Omega^{p,q}(E)\to \Omega^{p,q+1}(E)$. This makes sense since $\dim_{\C}X=1$, which implies that the integrability condition $\dbar_{E}\circ \dbar_{E}=0$ is always satisfied. For a holomorphic vector bundle $(E,\dbar_{E})$, whenever the complex structure is clear we will simply denote the space of $\dbar_{E}$-holomorphic global sections by $H^{0}(E)$.

We shall be interested in holomorphic vector bundles with structure group $\text{SL}(r,\C)$, i.e. holomorphic vector bundles $\mathcal{E}$ such that the fixed identification $\wedge^{r}E\rightarrow \underline{\C}$ induces $\wedge^{r}\mathcal{E}\xrightarrow{\cong}\mathcal{O}_{X}$. By an $\text{SL}(r,\C)$-Higgs bundle, we mean a pair $(\mathcal{E},\Phi)$ where $\Phi\in \Omega^{1,0}(\sllie(E))$ and $\dbar_{E}\Phi=0$. Here, $\sllie(E)$ is the bundle of traceless endomorphisms of $E$. For convenience, we will simply call it a Higgs bundle or a Higgs pair for short. We will specify a Higgs bundle by a holomorphic structure $\dbar_{E}$ and a \emph{Higgs field} $\Phi$, altogether denoted by the tuple $(\mathcal{E},\Phi)$, $(\mathcal{E},\dbar_{E},\Phi)$ or simply $(\dbar_{E},\Phi)$.

Let $\mathcal{A}=\mathcal{A}(E)$ be the space of Dolbeault operators defining holomorphic $\text{SL}(r,\C)$ vector bundles. It is an affine space modelled on $\Omega^{0,1}(\sllie(E))$. Define the \emph{configuration space} to be $\mathcal{C}:=\mathcal{A}\times \Omega^{1,0}(\sllie(E))$, and let $\mathcal{B}\subset \mathcal{C}$ be the space of Higgs bundles, i.e. $\mathcal{B}=\{(\dbar_{E},\Phi)\in \mathcal{C}: \dbar_{E}\Phi=0\}$. The \emph{(complex) gauge group} $\mathcal{G}_{\C}=\mathcal{G}_{\C}(E)$ consists of automophisms of $E$ fixing the determinant. It acts on $\mathcal{C}$ via conjugation: for $g\in \mathcal{G}_{\C}$ and $(\dbar_{E},\Phi)\in \mathcal{C}$, 
    \begin{equation}
g\cdot (\dbar_{E},\Phi):= (g\circ \dbar_{E},\circ g^{-1}, g\Phi g^{-1})=(\dbar_{E}-(\dbar_{E}g)g^{-1},g\Phi g^{-1}).        
    \end{equation}
This action preserves $\mathcal{B}$. For two Higgs bundles, we shall say they are gauge equivalent or isomorphic if they lie in the same (complex) gauge orbit.

The quotient $\mathcal{B}/\mathcal{G}_{\C}$ is badly behaved. For instance, it is not Hausdorff. Instead, we will look at the `moduli space' of semi-stable Higgs bundles. First, we recall the definitions of various notions of (slope) stability for $\text{SL}(r,\C)$-Higgs bundles. For $(\dbar_{E},\Phi)\in \mathcal{B}$, we say it is \emph{(semi)stable} if $\deg \mathcal{F}(\leq )< 0$ for all proper $\Phi$-invariant holomorphic subbundles $\mathcal{F}$ of $(E,\dbar_{E})$. A \emph{polystable} Higgs bundle is a direct sum of stable Higgs bundles of degree zero. Each semistable Higgs bundle admits a Jordan-H\"{o}lder filtration, whose associated graded gives rise to a polystable Higgs bundle which is unique up to isomorphism (cf. \cite{Sim92}). Let $\mathcal{B}^{ss}$ and $\mathcal{B}^{ps}$ be the locus of semistable and polystable Higgs bundles respectively. We define the moduli space of semistable $\text{SL}(r,\C)$-Higgs bundles to be
    \begin{equation}
\mathcal{M}_{H}=\mathcal{M}_{H}(\text{SL}(r,\C))=\mathcal{B}^{ss}\git \mathcal{G}_{\C}=\mathcal{B}^{ps}/\mathcal{G}_{\C}. 
    \end{equation}
Here, $\mathcal{B}^{ss}\git \mathcal{G}_{\C}$ refers to the quotient by $S$-equivalence classes (two semistable Higgs bundles are $S$-equivalent if their associated graded objects are gauge equivalent). We shall use $[(\dbar_{E},\Phi)]$ to denote the $S$-equivalence class of a semistable Higgs bundle $(\dbar_{E},\Phi)$. More generally, we will use $\mathcal{M}_{H}(r,d)$ to denote the moduli space of semistable Higgs bundles of rank $r$ and degree $d$. The corresponding notions of slope stability are the same, except that we replace $0$ by $d$.

The algebro-geometric construction of $\mathcal{M}_{H}(r,d)$ was carried at \cite{Nit91} in the case of smooth projective curves and in \cite{Sim94} in the case of smooth projective varieties. In the case where $X$ is a closed Riemann surface of genus at least two, the non-singular locus $\mathcal{M}^{s}_{H}$ of $\mathcal{M}_{H}$ is non-empty and consists of the isomorphism classes of stable Higgs bundles. In \cite{Hit87}, $\mathcal{M}^{s}_{H}$ was constructed by analytic means using the standard Kuranishi technique, and it was equipped with a hyperk\"{a}hler structure inherited from that of the configuration space.

To describe the hyperk\"{a}hler structure, we first fix a hermitian metric $h$ on the smooth vector bundle $E\to X$, whose induced metric $\det h$ on $\det E$ is compatible with the standard inner product on the trivial line bundle $\underline{\C}$. For $A\in \Omega^{k}(\text{End}(E))$, we denote by $$A^{*}=A^{*_{h}}$$
the $h$-hermitian adjoint of $A$. Often, when the prescribed hermitian metric $h$ is clear, we will omit the subscript $h$. The tangent space at each point in $\mathcal{C}$ can be identified with the space $\Omega^{0,1}(\sllie(E))\oplus \Omega^{1,0}(\sllie(E))$, which admits a hermitian inner product $H$ given by
    \begin{equation}
        H((\beta_1,\varphi_1),(\beta_2,\varphi_2)):= i\int_{X} \tr(\varphi_1 \wedge \varphi^{*}_2+\beta^{*}_{2}\wedge \beta_1)
    \end{equation}
with respect to the complex structure $I$ defined by
    \begin{equation}
        I(\beta,\varphi):=(i\beta,i\varphi).
    \end{equation}
The configuration space $\mathcal{C}$ can also be identified with the space of $\text{SL}(r,\C)$-connections in the following manner. Observe that every complex connection can be uniquely decomposed into an $h$-unitary connection and an $h$-hermitian one-form. The existence and uniqueness of Chern connections then identifies $\mathcal{A}$ with the space of $h$-unitary connections. Also, sending a hermitian one-form to its $(1,0)$-part gives an isomorphism $\Omega^{1}(i\mathfrak{su}(E))\xrightarrow{\cong}\Omega^{1,0}(\sllie(E))$, where $\mathfrak{su}(E)=\mathfrak{su}(E,h)$ is the bundle of traceless skew $h$-hermitian endomorphisms of $E$. Altogether, we have an identification between $\mathcal{C}$ and the space of $\text{SL}(r,\C)$ connections, with the latter modeled on $\Omega^{1}(\sllie(E))$. This gives another complex structure $J$ on $\mathcal{C}$:
    \begin{equation}
        J(\beta,\varphi):=(i\varphi^{*},-i\beta^{*}).
    \end{equation}
It is straightforward to verify that $(\text{Re}H, I,J,K=IJ)$ is a hyperk\"{a}hler structure on the configuration space $\mathcal{C}$. The K\"{a}hler form  with respect to $I$, $J$ and $K$ are denoted by $\omega_{I}$, $\omega_{J}$ and $\omega_{K}$ respectively. A direct computation then shows that the $I$-holomorphic symplectic form $\Omega_I$ is given by
    \begin{equation}\label{equation_holomorphic_symplectic_form}
        \Omega_I ((\beta_1,\varphi_1),(\beta_2,\varphi_2)) = \int_{X}\tr (\varphi_1\wedge \beta_2-\varphi_2\wedge \beta_1).
    \end{equation}

For $(\dbar_{E},\Phi)\in \mathcal{C}$, let $F_{(\dbar_{E},h)}\in \Omega^{2}(\mathfrak{su}(E,h))$ be the curvature of the Chern connection $\dbar_{E}+\partial^{h}_{E}$ associated with the complex structure defined by $\dbar_{E}$. The \emph{Hitchin equation}
    \begin{equation}
        \mu(\dbar_E, \Phi):=F_{(\dbar_{E},h)}+[\Phi,\Phi^{*}] =0
    \end{equation}
together with the holomorphicity condition
    \begin{equation}
        \dbar_{E}\Phi =0
    \end{equation}
can be intepreted as the hyperk\"{a}hler moment map with respect to the unitary gauge group $\mathcal{G}=\mathcal{G}(h)$, the subgroup of $\mathcal{G}_{\C}$ preserving the hermitian metric $h$. Developed in \cite{Hit87} and \cite{Don87}, and in \cite{Cor88} and \cite{Sim88} for the general case, the nonabelian Hodge correspondence (NAH for short) asserts that a Higgs bundle $(\dbar_{E},\Phi)$ is polystable (stable) if and only if there exists a metric $h$ that solves the Hitchin equation, in which case the connection $$NAH(\dbar_{E},\Phi)=\dbar_{E}+\partial^{h}_{E}+\Phi+\Phi^{*_{h}}$$
is flat and completely reducible (irreducible). Such a connection is unique. Conversely, a flat connection $D$ is completely reducible (irreducible) if and only if there exists a metric $h$ such that in the decomposition $d_{A}+\Psi$ where $d_A$ is $k$-unitary and $\Psi$ is $h$-hermitian, we have
$$d_{A}\Psi^{1,0}=0,$$
in which case $\dbar_{E}=d_{A}''$, the $(0,1)$-part of $d_A$, and $\Phi=\Psi^{1,0}$ defines a polystable (stable) Higgs bundle $(\dbar_{E},\Phi)$. In both directions, the metric $h$ is called a \emph{harmonic metric} to $(\dbar_{E},\Phi)$ or $D$. The automorphism group $G$ of a polystable Higgs bundle acts transitively on the set of its harmonic metrics. If $h$ is a harmonic metric for two complex gauge-equivalent polystable Higgs bundles, then they are $h$-unitary gauge-equivalent. Globally, the NAH yields the homeomorphism $$\mathcal{M}_{H}\cong (\mathcal{B}^{ps}\cap\mu^{-1}(0))/\mathcal{G}.$$

In this paper, we will often need to consider the local geometry around a point $[(\dbar_{E},\Phi)]$ in $\mathcal{M}_{H}$ (more specifically, a $\C^{*}$-fixed point). When $(\dbar_{E},\Phi)$ is stable, a slice consisting of stable Higgs bundles which represents an open neighborhood around $[(\dbar_{E},\Phi)]$ can be obtained using the standard Kuranishi slice technique. However, when $[(\dbar_{E},\Phi)]$ is a singular point in $\mathcal{M}_{H}$, which is the case we are concerned with, the construction of a local model for an open neighborhood of $[(\dbar_{E},\Phi)]$ is more delicate. We will utilize the analytic local model constructed rigorously in \cite{Fan22} which is based on properties of the Yang-Mills-Higgs flow studied in \cite{Wil08}. The next subsection will summarize the construction of this local model. 

\subsection{The local Kuranishi model around a polystable Higgs bundle}

Let $(\dbar_{E},\Phi)$ be a strictly polystable $\text{SL}(r,\C)$-Higgs bundle. The deformation complex of $(\dbar_{E},\Phi)$ is given by
    \begin{equation}\label{equation_deformation_complex}
C^{\bullet}=C^{\bullet}(\dbar_{E},\Phi): \Omega^{0}(\sllie(E))\xrightarrow{D''}\Omega^{0,1}(\sllie(E))\oplus\Omega^{1,0}(\sllie(E))\xrightarrow{D''} \Omega^{1,1}(\sllie(E))
    \end{equation}
where $D'':=\dbar_{E}+\Phi$. For the degree one piece of the complex $C^{1}=\Omega^{0,1}(\sllie(E))\oplus\Omega^{1,0}(\sllie(E))$, we usually use a tuple $u=(\beta,\varphi)$ to denote its elements. We will also use $u$ to represent $\beta+\varphi\in \Omega^{1}(\sllie(E))$. For any $v\in \Omega^{1}(\sllie(E))$, we write $v=v^{0,1}+v^{1,0}$ to denote its type decomposition.  

Fix a harmonic metric $h$ of $(\dbar_{E},\Phi)$. We define $D':=\partial^{h}+\Phi^{*}$ where $\dbar_{E}+\partial^{h}$ is the Chern connection with respect to the metric $h$, $\Phi^{*}$ the $h$-hermitian adjoint of $\Phi$, and $D:=D'+D''$ is the flat connection under NAH. We remark that $D'$ is independent of the choice of harmonic metrics of $(\dbar_{E},\Phi)$. These operators satisfy the following K\"{a}hler identities (for example, see \cite{Sim92}):
\begin{equation}
    (D'')^{*} = -i[\Lambda, D'] \quad \text{and} \quad (D')^{*} = i[\Lambda, D''].
\end{equation}
We denote the space of harmonic forms of the deformation complex by
    \begin{equation}
        \mathcal{H}^{i}=\mathcal{H}^{i}({C^{\bullet}}):=\{x\in C^{i}: D''x=0, (D'')^{*}x = 0\}.
    \end{equation}
The zero-th cohomology space $\mathcal{H}^{0}=\ker D''$ can be identified with the Lie algebra of the automorphism group of $(\dbar_{E},\Phi)$. Moreover, via Serre duality we have $\mathcal{H}^{2}\cong (\mathcal{H}^{0})^{\vee}$. When $(\dbar_{E},\Phi)$ is strictly polystable, $\mathcal{H}^{0}\neq \{0\}$. Denote by $G$ the automorphism group of $(\dbar_{E},\Phi)$, i.e. its stabilizer group under the complex gauge group action. Then $G$ acts on $\mathcal{H}^{i}$, $i=0, 1, 2$, by conjugation. It has a maximal compact subgroup $K=K(h)$ which consists of elements in $G$ fixing the metric $h$. All other harmonic metrics of $(\dbar_{E},\Phi)$ can be obtained from pulling back by elements in $G$. In general, for a gauge transformation $g\in \mathcal{G}_{\C}$, we shall use $g^{*}h$ to refer to the pullback metric defined by $g^{*}h(v,w)=h(gv,gw)$.

In the remainder of this subsection, we will recall the local model around $[(\dbar_{E},\Phi)]\in \mathcal{M}_{H}$ constructed in Section 3 of \cite{Fan22}. Define the \emph{Kuranishi map} $\kappa: \Omega^{0,1}(\sllie(E))\oplus \Omega^{1,0}(\sllie(E))\to \Omega^{0,1}(\sllie(E))\oplus \Omega^{1,0}(\sllie(E))$ associated with $(\dbar_{E},\Phi)$ by 
    \begin{equation}\label{kuranishi_map}
\kappa (\beta,\varphi)  = (\beta,\varphi)+(D'')^{*}\Gamma([\beta,\varphi]),
    \end{equation}
where $\Gamma$ is the Green's operator of the Laplacian $D''(D'')^{*}$ on $\Omega^{1,1}(\sllie(E))$. The automorphism group $G$ acts on both sides by conjugation, and $\kappa$ is $G$-equivariant. We remark that this $G$-action is compatible with the action of the gauge group: for $g\in G$, $g\cdot (\dbar_{E}+\beta,\Phi+\varphi)=(\dbar_{E}+g\beta g^{-1},\Phi+g\varphi g^{-1})$. Let $$\tilde{Z}:=\{(\beta,\varphi)\in \ker D': D''(\beta,\varphi)+[\beta,\varphi]\perp \mathcal{H}^{2}\}.$$
Restricted to $\tilde{Z}$, $\kappa$ takes values in $\mathcal{H}^{1}$. Moreover, $\kappa$ is locally invertible at the origin. There exists an open subset $\tilde{U}$ of $\tilde{Z}$ containing the origin, such that $\kappa|_{\tilde{U}}$ is a local biholomorphism onto an open set $U=\kappa(\tilde{U})$ of $\mathcal{H}^{1}$ containing $0$. By shrinking $\tilde{U}$ if needed, we may take $U$ to be an open ball centered at $0$. Then, the restriction $\kappa|_{\tilde{U}}$ is $K$-equivariant. Let $\kappa^{-1}: U\to \Omega^{0,1}(\sllie(E))\oplus \Omega^{1,0}(\sllie(E))$ be the inverse of $\kappa$. Let
    \begin{equation}\label{equation_Z_definition}
Z:=\{(\beta,\varphi)\in \ker D': D''(\beta,\varphi)+[\beta,\varphi]=0\}\subset \tilde{Z}.
    \end{equation}
Note that $(\dbar_{E},\Phi)+Z\subset \mathcal{B}$. For $u=(\beta,\varphi)\in Z$, we will use interchangeably $$(\dbar_{u},\Phi_u)=(\dbar_{E},\Phi)+u=(\dbar_{E}+\beta,\Phi+\varphi)$$
to denote the Higgs bundle translated from $(\dbar_{E},\Phi)$ by $u$.

Define $k: \mathcal{H}^{1}\to \mathcal{H}^{2}$ by
    \begin{equation}
k(x):=\dfrac{1}{2}P[x,x],
    \end{equation}
where $P$ is the projection to $\mathcal{H}^{2}$. Let $\mathcal{Z}:=k^{-1}(0)\subset \mathcal{H}^{1}$. We remark that (cf. Theorem 5.1, \cite{Fan22}) for $(\beta,\varphi)\in \tilde{Z}$, $$D''(\beta,\varphi)+[\beta,\varphi]=k(\kappa(\beta,\varphi)),$$
from which it follows that $\kappa(Z)\subset \mathcal{Z}$. By local biholomorphicity, $\kappa(Z\cap\tilde{U})=\mathcal{Z}\cap U$ for $U$ sufficiently small. We will refer to either one of $Z$, $\mathcal{Z}$, $Z\cap \tilde{U}$, $\kappa(Z)\cap U$, or their restrictions to smaller open neighborhoods as the \emph{Kuranishi space}. For $U$ sufficiently small, we have the map $\theta:=\kappa^{-1}: \kappa(Z)\cap U\to \mathcal{B}^{ss}$, defined by 
    \begin{equation}\label{equation_kuranishi_inverse_theta}
\theta(x)=(\dbar_{E}+(\kappa^{-1}(x))^{0,1},\Phi+(\kappa^{-1}(x))^{1,0}).        
    \end{equation}
As $\mathcal{H}^{1}$ is a $G$-representation, there is a corresponding notion of stability for vectors in $\mathcal{H}^{1}$. One key step in the construction in \cite{Fan22} is to relate polystability in $\mathcal{H}^{1}$ with polystability of the corresponding Higgs bundles, for vectors in $\kappa (Z)\cap U$ with $U$ sufficiently small. Recall that for a $G$-representation $V$, a nonzero vector $v$ is said to be \emph{(G-)polystable} if its orbit $G\cdot v$ is closed in $V$. If $v$ is polystable and its isotropy group $G_v$ is finite, we say $v$ is \emph{(G-)stable}.

We are now ready to describe the construction of the local model around $[(\dbar_{E},\Phi)]$ in \cite{Fan22}. For a sufficiently small open ball $U$ around the origin, $G\cdot (\mathcal{Z}\cap U)$ is a closed complex subspace of the open set $G\cdot U$, and the Kuranishi space $\mathcal{Z}\cap U$ is open in $G\cdot (\mathcal{Z}\cap U)$. We have the following local parametrization of $\mathcal{M}_{H}$ around $[(\dbar_{E},\Phi)]$.

\begin{proposition}[\cite{Fan22}, Section 3]\label{proposition_yue_local_model}
For a sufficiently small ball $U$ in $\mathcal{H}^{1}$ centered at $0$, the continuous map $\mathcal{Z}\cap U\xrightarrow{\theta}\mathcal{B}^{ss}\longrightarrow \mathcal{M}_{H}$ induces an open embedding $$G\cdot (\mathcal{Z}\cap U)\git G\xlongrightarrow{\pi} \mathcal{M}_{H}.$$
In fact, if $x\in \mathcal{Z}\cap U$ is polystable, then $\theta(x)$ is a polystable Higgs bundle.
\end{proposition}
Here, $G\cdot (\mathcal{Z}\cap U)\git G$ denotes the analytic GIT quotient of the reduced complex space $G\cdot (\mathcal{Z}\cap U)$.

\subsection{Stability in linear representations}\label{section_linear_rep}

The purpose of this subsection is to collect some useful criteria for stability in a $G$-representation $V$ of a connected complex reductive group $G$, and to set up the notation to be used in the last section in the case when $G$ is a complex torus. First, we recall the theorem of Kempf and Ness \cite{KN79}, which relates the polystability and stability of a nonzero vector $v\in V$ in terms of the existence of minimizers of the length squared functional. Then, we specialize to the case when $G$ is a torus, where $V$ admits a weight decomposition. In this case, stability of $v$ can be detected from the affine properties of the convex hull of its effective weights (see Theorem 9.2 in \cite{Dol03} for example).

Let $G$ be a connected complex reductive group with a maximal compact subgroup $K$ and $V$ a complex $G$-representation with a hermitian inner product such that $K$ acts unitarily. Recall that we have the decomposition $G=K\text{exp}(i\mathfrak{k})$, where $\mathfrak{k}=\text{Lie}(K)$. Now, let $v\in V-\{0\}$. The \emph{Kempf-Ness functional} of $v$, denoted by $\ell_v: i\mathfrak{k}\to \R$, is given by
    \begin{equation}\label{equation_kempf_ness_functional}
        \ell_v(x):= ||e^{x}\cdot v||^2.
    \end{equation}

\begin{theorem}[\cite{KN79}, Kempf-Ness]
A nonzero vector $v$ is polystable if and only if $\inf \ell_v$ can be attained, and is stable if and only if the minimizer is unique.    
\end{theorem}

Next, we will specialize to the case where $G$ is an $k$-dimensional complex torus, i.e. $G\cong (\C^{\times})^{k}$. The setup and results in the rest of this subsection will not be used until the last section, and hence can be safely skipped until then.

The representation space $V$ admits an orthogonal decomposition into weight spaces
    \begin{equation}
        V=\bigoplus_{\lambda\in \Lambda_G} V_{\lambda}
    \end{equation}
where $\Lambda_G\subset X(G):=\text{Hom}(G,\C^{*})$ is the set of weights of the $G$-representation. We identify $X(G)\otimes_{\Z}\R$ with $(i\mathfrak{k})^{\vee}\cong \R^{k}$ so that we can regard $X(G)\subset (i\mathfrak{k})^{\vee}$ as an integral lattice of the dual space $(i\mathfrak{k})^{\vee}$.

Let $v\in V-\{0\}$. For $\lambda\in \Lambda_G$, we let $v_{\lambda}$ be the projection of $v$ to the $\lambda$-weight space. Let
    \begin{equation}\label{equation_effective_weight}
\Lambda_{G}(v):=\{\lambda\in \Lambda_{G}: v_{\lambda}\neq 0\}.      
    \end{equation}
We will call the set $\Lambda_{G}(v)$ the \emph{effective weight} of $v$ with respect to the $G$-action. Finally, let $C(v)=\text{Conv}(\Lambda_{G}(v))$ be the closed convex hull of $\Lambda_{G}(v)$ in $(i\mathfrak{k})^{\vee}$.

In the last section, we will need the following criterion of the stability of $v$ in terms of properties of $C(v)$, which is essentially a visual depiction of the Hilbert-Mumford criterion. It can be found, for example, in Theorem 9.2 of \cite{Dol03}. For completeness, we will give the proof of the equivalence between this criterion and the existence of unique minimizer of the Kempf-Ness functional.
\begin{proposition}\label{prop-convexhull}
The following are equivalent:
    \begin{enumerate}[label=(\roman*)]
        \item
$v$ is stable
        \item
$\ell_v$ has a unique minimizer.
        \item
$0\in C(v)^{\circ}$, the interior of $C(v)$ (in the euclidean topology of $\R^{k}$).
    \end{enumerate}
\end{proposition}
\begin{proof}
The equivalence between (i) and (ii) is just a special case of the Kempf-Ness theorem. In terms of the weight space decomposition, we can re-write (\ref{equation_kempf_ness_functional}) as
    \begin{equation}
\ell_v(x) = \sum_{\lambda\in \Lambda_{G}(v)} ||v_\lambda||^2 e^{2\langle \lambda, x\rangle}. 
    \end{equation}

First, we assume (ii). Note that for any $v'\in G\cdot v$, $\Lambda_{G}(v')=\Lambda_{G}(v)$. Hence, we can assume without loss of generality that the unique minimizer takes place at the origin. Since $\Lambda_{G}(v)$ is finite, $C(v)$ is a convex polytope. If $k=1$, then $C(v)$ is a closed interval of the form $[-n,m]$ where $n,m\in \Z_{>0}$. Now, assume $k\geq 2$ We first claim that $C(v)$ is $k$-dimensional. Suppose to the contrary that $\dim C(v)<k$. Then $C(v)$ is contained in some hyperplane $H=\{v\in (i\mathfrak{k})^{\vee}: \langle v,u\rangle = a\}$ for some $u\in i\mathfrak{k}$ and $a\leq 0$. For $t\geq 0$, $\ell_v (tu)=\sum ||v_\lambda||^2 e^{2ta}\leq \ell_v(0)$ . This contradicts the assumption that $0$ is the unique minimizer of $\ell_v$. A similar argument together with the hyperplane separation theorem will show that $0\in C(v)^{\circ}$.

Conversely, assume (iii). First, we claim that $\inf \ell_v$ can be attained. Let $S$ be the unit sphere in $i\mathfrak{k}$ determined by some inner product. Pick a minimizing sequence $\{y_j\}_{j\geq 0}$. Without loss of generality, we may assume the sequence is unbounded. Hence, passing to subsequences if necessary, we may write $y_j=r_j u_j$ where $0<r_j\nearrow \infty$, $u_j\in S$ and $u_j\to u\in S$. Now, we evaluate $$\ell_v (y_j) = \sum ||v_\lambda||^2 (e^{2\langle \lambda, u_j\rangle})^{r_j}\to \inf \ell_v.$$
Since $\ell_v(y_j)$ is bounded, we have $\langle \lambda,u\rangle \leq 0$ for all $\lambda \in \Lambda_{G}(v)$. As $C(v)$ is the convex hull of $\Lambda_{G}(v)$, it implies $\langle C(v),u\rangle \leq 0$. This contradicts to the assumption $0\in C(v)^{\circ}$ since any open ball centered at $0$ contains some point $y$ such that $\langle y,u\rangle >0$. Therefore, no unbounded minimizing sequences exist and thus $\inf \ell_v$ can be attained.

Let $x_0$ be a minimizer of $\ell_v$. For any $w\neq 0$, consider the function $$t\mapsto \ell_v(x_0+tw)=\sum ||v_\lambda||^2 e^{2\langle \lambda, x_0\rangle}e^{2\langle \lambda, w\rangle t},$$ which is convex and has a global minimum at $t=0$. Moreover, it is either strictly convex and thereby uniquely minimized, or it is constant. In the latter case, we have $w\in D$, where $$D=\bigcap_{\lambda\in \Lambda'(v)} \ker\lambda.$$

It is also easy to see that for any $w\in D$, $t\mapsto \ell_v (x_0+tw)\equiv \ell_v (x_0)$. Therefore, the set of minimizers of $\ell_v$ is $x_0+D$. But $\langle\Lambda_{G}(v), D\rangle \equiv 0$ by definition, so $C(v)\subset D^{\perp}:=\{v\in (i\mathfrak{k})^{\vee}: \langle v, D\rangle\equiv 0\}$. As $0\in C(v)^\circ$ by assumption, we must have $D^{\perp}=(i\mathfrak{k})^{\vee}$. It follows that $D=\{0\}$ and thus $x_0$ is the unique minimizer.
\end{proof}

%% file: model.tex
\section{The central locus}\label{section_slice_S_kuranishi_model}

The aim of this section is to prove Theorem \ref{theorem_lagrangian}. Given a $\C^{*}$-fixed point $[(\dbar_{E},\Phi)]\in \mathcal{M}_{H}$, we will define and construct a model for the central locus of the upward flow through $[(\dbar_{E},\Phi)]$. In the case where $(\dbar_{E},\Phi)$ is stable, this construction is the same as the global parametrization of the upward flow through $[(\dbar_{E},\Phi)]$ obtained in \cite{CW19}. When $(\dbar_{E},\Phi)$ is strictly polystable, however, we need to take into account the non-trivial automorphism group $G$ of $(\dbar_{E},\Phi)$. Another difference compared with the stable case is that there is an ambiguity in decomposing $(\dbar_{E},\Phi)$ into a system of Hodge bundles (to be recalled below), a step that is needed to decompose the deformation complex of $(\dbar_{E},\Phi)$ into graded pieces.

In the following, We begin by recalling basic facts about the $\C^{*}$-action on $\mathcal{M}_{H}$. We then recall the notion of system of Hodge bundles and prove a couple of auxiliary results. In the next subsection, we decompose $(\dbar_{E},\Phi)$ into to a system of Hodge bundles in a specific way, so that $G$ respects the induced graded structure on the deformation complex of $(\dbar_{E},\Phi)$. This allows us to define the central locus $S$ of the upward flow through $[(\dbar_{E},\Phi)]$. We then show that $S\cap \mathcal{M}_{H}^{s}$ is Lagrangian (Proposition \ref{proposition_central_locus_lagrangian}). Finally, under Assumption \ref{assumption_I} or Assumption \ref{assumption_II}, we prove that $S\cap \mathcal{M}^{s}_{H}$ is non-empty (Theorem \ref{theorem_non_empty}). Together, this proves Theorem \ref{theorem_lagrangian}.

\subsection{$\C^{*}$-fixed points in $\mathcal{M}_{H}$ and system of Hodge bundles (SHB)}
The moduli space $\mathcal{M}_{H}$ admits a $\C^{*}$-action $t\cdot [(\dbar_{E},\Phi)]\mapsto [(\dbar_{E},t\Phi)]$. For $[(\dbar_{E},\Phi)]\in \mathcal{M}_{H}$, as a consequence of the properness of the Hitchin fibration, the limit $$\lim_{t\to 0}[(\dbar_{E},t\Phi)]$$
always exists, and we shall call it the \emph{downward limit} of $[(\dbar_{E},\Phi)]$. Alternatively, the existence of downward limits can be deduced from the constructions in \cite{Sim10}.

By a $\C^{*}$-fixed point, we mean a point in $\mathcal{M}_{H}$ fixed by the $\C^{*}$-action. Downward limits are always $\C^{*}$-fixed points. A polystable representative of a $\C^{*}$-fixed point can be written as a \emph{system of Hodge bundles}. Recall that a system of Hodge bundles (SHB for short) is a direct sum of holomorphic vector bundles $(E,\dbar_{E})=\bigoplus_{a=1}^{k}(E_{a},\dbar_{a})$ together with a Higgs field $\Phi=\sum_{a=1}^{k-1}\Phi_{a}$, where $\Phi_{a}\in \Omega^{1,0}(\text{Hom}(E_{a},E_{a+1})$ and $\dbar_{E}\Phi_{a}=0$. When an SHB is semistable as a Higgs bundle, one can show that $[(\dbar_{E},\Phi)]$ is a $\C^{*}$-fixed point. Conversely, if $(\dbar_{E},\Phi)$ is polystable and $[(\dbar_{E},\Phi)]$ is a $\C^{*}$-fixed point, then $(\dbar_{E},\Phi)$ can be endowed with an SHB structure, which is unique when $(\dbar_{E},\Phi)$ is stable. The summands $E_{a}$ are mutually orthogonal with respect to the harmonic metric of $(\dbar_{E},\Phi)$.

In this article, there is one SHB structure for a polystable $(\dbar_{E},\Phi)$ that is convenient for our purposes. Its construction will be carried out in the next subsection. In preparation for that, we first establish the following

\begin{proposition}\label{proposition_shb_block}
Suppose $(\dbar_{E},\Phi)$ is polystable and $[(\dbar_{E},\Phi)]\in \mathcal{M}_{H}$ is a $\C^{*}$-fixed point. Then $(\dbar_{E},\Phi)$ can be written as a direct sum of stable Higgs bundles which represent $\C^{*}$-fixed points.
\end{proposition}
\begin{proof}
First, we claim that all $t\in \C^{*}$, $(\dbar_{E},t\Phi)$ lie in the same gauge equivalence class. This is true for $t\in S^{1}$ since $[(\dbar_{E},\Phi)] = [(\dbar_{E},t\Phi)]$ and the same harmonic metrics solve the Hitchin equation for each $(\dbar_{E},t\Phi)$. Pick a $t$, $|t|=1$, which is not a root of unity. Then there exists an isomorphism $f_{t}$ from $(\dbar_{E},\Phi)$ to $(\dbar_{E},t\Phi)$. In particular, $f_t$ is holomorphic with respect to the holomorphic structure defined by $\dbar_{E}$. Proved in Lemma 4.1 of \cite{Sim92}, the generalized eigenspaces of $f_{t}$ give $(\dbar_{E},\Phi)$ a system of Hodge bundles structure, with which we can construct isomorphisms to all $(\dbar_{E},t\Phi)$, $t\in \C^{*}$.

Next, we decompose $(\dbar_{E},\Phi)$ into a direct sum of stable summands $\bigoplus_{i=1}^{k}(\dbar_{E^i},\Phi^{i})$. We claim each $(\dbar_{E^i},\Phi^{i})$ is a $\C^{*}$-fixed point in the moduli space of degree zero polystable Higgs bundles of the corresponding ranks. Pick a non-zero scalar $t\in \C^{*}$ which is not a root of unity. Each $(\dbar_{E^i},t\Phi^i)$ is stable and we have $(\dbar_{E},\Phi)\cong (\dbar_{E},t\Phi)=\bigoplus_{i}(\dbar_{E^i},t\Phi^{i})$. It follows that $(\dbar_{E^i},t\Phi^{i})\cong (\dbar_{E^{\sigma(i)}},\Phi^{\sigma(i)})$ for some permutation $\sigma$ of $\{1,\cdots, k\}$. Composing these isomorphisms, we have $(\dbar_{E^{i}},t^{j}\Phi^{i})\cong (\dbar_{E^{\sigma^{j}(i)}},\Phi^{\sigma^{j}(i)})$. Now, setting $j=\text{ord}(\sigma)$, we have $(\dbar_{E^i},\Phi^{i})\cong (\dbar_{E^i},t^{j}\Phi^{i})$. Applying the procedure in Lemma 4.1 of \cite{Sim92} again, we see that each $[(\dbar_{E^{i}},\Phi^{i})]$ is a $\C^{*}$-fixed point.
\end{proof}

We will also need the following lemma, which roughly speaking says that taking downward limits commutes with direct sums.

\begin{lemma}\label{lemma_limit_direct_sum_commute}
Let $(\mathcal{E},\dbar_{E},\Phi)$ be a polystable Higgs bundle. Suppose $(\mathcal{E},\dbar_{E},\Phi)=\bigoplus_{i=1}^{k} (\mathcal{E}^{i},\dbar^{i},\Phi^{i})$, where each summand $(\mathcal{E}^{i},\Phi^{i})$ is a stable Higgs bundle with downward limit $\lim_{t\to 0}[(\mathcal{E}^{i},\dbar^{i},t\Phi^{i})]=[(\mathcal{E}^{i}_{0},\dbar^{i}_{0},\Phi^{i}_{0})]$ represented by $(\dbar^{i}_{0},\Phi^{i}_{0})$ which is polystable. Then $\lim_{t\to 0}[(\mathcal{E},\dbar_{E},t\Phi)]=[\bigoplus_{i=1}^{k}(\mathcal{E}^{i}_{0},\dbar^{i}_{0},\Phi^{i}_{0})]$.
\end{lemma}
\begin{proof}
For each $i$, using Proposition \ref{proposition_shb_block}, we put an SHB structure on $(\dbar^{i}_{0},\Phi^{i}_{0})$ in the form \eqref{equation_hodge_summand_decomposition} and fix a harmonic metric $h_i$. Let $G_i$ be the automorphism group of $(\dbar^{i}_{0},\Phi^{i}_{0})$, $U_i$ a small ball around $0$ in $\mathcal{H}^{1}$ of the deformation complex of $(\dbar^{i}_{0},\Phi^{i}_{0})$ as in Proposition \ref{proposition_yue_local_model}, $\mathcal{Z}_{i}$ the Kuranishi space and let $$\pi_{i}: G_i \cdot (\mathcal{Z}_{i}\cap U_i)\git G_i\to \mathcal{M}_{H}(r_i,0)$$
be the local Kuranishi model into the moduli space of semistable Higgs bundles of rank $r_i$ and degree zero, where $r_i=\text{rk}(\mathcal{E}^{i})$. For small $|t|$, every $[(\dbar^{i},t\Phi^{i})]$ has representatives $x_i^{t}$ in $U_i$, i.e. $\pi_{i}([x_i^{t}])=[(\dbar^{i},t\Phi^{i})]$. Shrinking the radius of each $U_i$ to zero, we see that there exists a sequence $t_n\to 0$ such that $x_{i}^{t_n}\to 0$ for all $i$. For each $x_{i}^{t_n}$, choose $y_{i}^{t_n}\in \overline{G_i\cdot x_{i}^{t_n}}$ which minimizes the $h_i$-norm along $\overline{G_i\cdot x_{i}^{t_n}}$. Then $y_{i}^{t_n}\in U_i$, $y_i^{t_n}\to 0$ as $n\to 0$ and the $G_i$-orbit of each $y_{i}^{t_n}$ is closed. By Proposition \ref{proposition_yue_local_model}, $\theta(y_{i}^{t_{n}})$ is gauge-equivalent to $(\dbar^{i},t_n\Phi^{i})$ (for convenience, here we use the same notation $\theta$ to denote the inverses of the Kuranishi maps with respect to all $(\dbar^{i}_0,\Phi^{i}_0)$). It follows that $(\dbar_{E},t_n \Phi)$ is gauge equivalent to $\bigoplus_{i=1}^{k} \theta(y_i^{t_n})$. For each $n$, we may take the gauge transformations to lie in $\text{SL}$. Hence, $\lim_{n\to 0}[(\dbar_{E},t_{n}\Phi)]=[\bigoplus_{i=1}^{k} (\dbar^{i}_{0},\Phi^{i}_{0})]$. Since $\lim_{t\to 0}[(\dbar_{E},t\Phi)]$ exists, the result follows.
\end{proof}

\subsection{The deformation complex at a polystable SHB}\label{subsection_SHB_deformation}

Now, suppose $[(\dbar_{E},\Phi)]$ is a $\C^{*}$-fixed point, with $(\dbar_{E},\Phi)$ polystable. Let $G$ be its automorphism group. We will specify an SHB structure on $(\dbar_{E},\Phi)$ in the following manner. By Proposition \ref{proposition_shb_block}, there exists a holomorphic decomposition into $k$ stable summands
    \begin{equation}\label{equation_stable_summand_decomposition}
(E,\dbar_{E}) =\bigoplus_{i=1}^{k}(E^{i}, \dbar^{i}) \quad \text{ and }\quad \Phi=\sum_{i=1}^{k}\Phi^{i},        
    \end{equation}
where $\Phi^{i}\in H^{0}(\text{End}(E^{i})\otimes K_{X})$, such that each $(\dbar^{i},\Phi^{i})$ is a stable Higgs bundle of degree zero representing a $\C^{*}$-fixed point. Each $(\dbar^{i},\Phi^{i})$ admits a unique SHB structure given by
    \begin{equation}
(E^{i},\dbar^{i})=\bigoplus_{a=1}^{\ell_{i}}(E^{i}_{a},\dbar^{i}_{a}) \quad\text{ and } \quad\Phi^{i}=\sum_{a=1}^{\ell_i-1}\Phi^{i}_{a}       
    \end{equation}
where $\Phi^{i}_{a}\in H^{0}(\text{Hom}(E^{i}_{a},E^{i}_{a+1})\otimes K_{X})$. Let
    \begin{equation}\label{equation_hodge_summand_decomposition}
(E_{a},\dbar_{a}):=\bigoplus_{i} (E^{i}_{a},\dbar^{i}_{a}) \quad \text{and} \quad \Phi_{a}:=\sum_{i} \Phi^{i}_{a}
    \end{equation}
so that $\Phi_{a}\in H^{0}(\text{Hom}(E_{a},E_{a+1})\otimes K_{X})$. The direct sum $(E,\dbar_{E})=\bigoplus_{a=1}^{\ell}(E_{a}, \dbar_{a})$, where $\ell=\max_{i}\{\ell_{i}\}$. \eqref{equation_hodge_summand_decomposition} then defines an SHB structure on $(\dbar_{E},\Phi)$. We can take a harmonic metric $h$ of $(\dbar_{E},\Phi)$ such that the summands $E^{i}$ and $E^{j}$ are mutually orthogonal with respect to $h$. Thus, $E_{a}$ and $E_{b}$ are mutually orthogonal. Henceforth, this is be the SHB structure we will refer to. We remark that the resulting Hodge summands \eqref{equation_hodge_summand_decomposition} are independent of the choice of the stable summand decomposition \eqref{equation_stable_summand_decomposition}.

Next, we recall the deformation complex \eqref{equation_deformation_complex}. As in \cite{CW19}, we introduce a grading on $\text{End}(E)$, whose $\ell$-th component is given by
    \begin{equation}
\text{End}_{\ell}(E):=\bigoplus_{i-j=\ell} \text{Hom}(E_{i},E_{j}).    
    \end{equation}
There is a $\C^{*}$-action on $C^{\bullet}$ defined in the following way. Let $w\in \Omega^{0}(\sllie(E))$, $(\beta,\varphi)\in \Omega^{0,1}(\sllie(E))\oplus\Omega^{1,0}(\sllie(E))$ and $v\in \Omega^{1,1}(\sllie(E))$. We will add a subscript $j$ to denote the $j$-th graded component with respect to the grading on $\text{End}(E)$. For $t\in \C^{*}$, the action is given by
    \begin{align}\label{equation_C*_action_forms}
        t\cdot w &:= \sum_{j} t^{j}w_j\\
        t\cdot (\beta,\varphi) &:= \bigg( \sum_{j} t^{j}\beta_{j},\sum_{j}t^{j+1}\varphi_{j} \bigg) \\
        t\cdot v &:= \sum_{j} t^{j+1} v_{j}.
    \end{align}

It is a direct computation to verify that $D''$ and $(D'')^{*}$ are $\C^{*}$-equivariant. The deformation complex is then decomposed into subcomplexes
    \begin{equation}
(C^{\bullet})_{j}: \Omega^{0}(\sllie(E)_{j})\xrightarrow{D''} \Omega^{0,1}(\sllie(E)_{j})\oplus \Omega^{1,0}(\sllie(E)_{j-1})\xrightarrow{D''} \Omega^{1,1}(\sllie(E)_{j-1}).
    \end{equation}
The space of harmonic forms are decomposed accordingly: $\mathcal{H}^{i} =\bigoplus_{j} (\mathcal{H}^{i})_{j}$ where $(\mathcal{H}^{i})_{j}=\mathcal{H}^{i}((C^{\bullet})_{j})$.

\begin{lemma}\label{lemma_h0_blockdiagonal}
\begin{enumerate}[label=(\roman*)]
    \item
$\mathcal{H}^{0}=(\mathcal{H}^{0})_{0}$ and $\mathcal{H}^{2}=(\mathcal{H}^{2})_{1}$.
    \item
The automorphism group $G$ acts by graded degree zero.
\end{enumerate}

\end{lemma}
\begin{proof}
Elements of $\mathcal{H}^{0}$ are endomorphisms of $(\dbar_{E},\Phi)$. Decompose them into $\text{Hom}((\dbar^{i},\Phi^{i}),(\dbar^{j},\Phi^{j}))$. Notice that $\text{Hom}((\dbar^{i},\Phi^{i}),(\dbar^{j},\Phi^{j}))$ is trivial when $(\dbar^{i},\Phi^{i})$ and $(\dbar^{j},\Phi^{j})$ are not equivalent and is one-dimensional otherwise. In the latter case, the non-zero elements preserve the SHB structures. Putting them back together, we see that elements in $\mathcal{H}^{0}$ preserve each $E_{i}$. This proves (ii) and the case for $\mathcal{H}^{0}$ in (i). The case for $\mathcal{H}^{2}$ follows from duality.
\end{proof}

Now, let $C^{\bullet}_{+}:=\bigoplus_{j>0}(C^{\bullet})_{j}$ and $\mathcal{H}^{1}_{+}:=\mathcal{H}^{1}(C^{\bullet}_{+})=\bigoplus_{j>0}(\mathcal{H}^{1})_{j}$. The computation in Lemma 3.6 of \cite{CW19} together with (i) of Lemma \ref{lemma_h0_blockdiagonal} prove the following

\begin{proposition}\label{proposition_dim}
$\chi(C^{\bullet}_{+})=\dim\mathcal{H}^{2}-\dim \mathcal{H}^{1}_{+}=-(r^2-1)(g-1)=-\dfrac{1}{2}\dim \mathcal{M}_{H}$.
\end{proposition}

\subsection{Model for the central locus}

Associated with the polystable SHB $(\dbar_{E},\Phi)$, we recall the Kuranishi map \eqref{kuranishi_map} $$\kappa (\beta,\varphi)  = (\beta,\varphi)+(D'')^{*}\Gamma([\beta,\varphi]).$$
In addition to the $G$-equivariance, $\kappa$ is also $\C^{*}$-equivariant with respect to the $\C^{*}$-action on $\Omega^{0,1}(\sllie(E))\oplus \Omega^{1,0}(\sllie(E))$ defined in the previous subsection. Moreover, the two actions commute with each other.

Define the space
    \begin{equation}\label{equation_S+}
S_{+}=S_{+}(\dbar_{E},\Phi):= C^{1}_{+}\cap Z=C^{1}_{+}\cap\{(\beta,\varphi): D''(\beta,\varphi)+[\beta,\varphi]=D'(\beta,\varphi)=0\},
    \end{equation}
which is a closed complex subspace of $Z$. By examining the graded pieces, we see that the restriction of $\kappa$ to $S_{+}$ takes values in $\mathcal{H}^{1}_{+}$. Also, note that both $S_+$ and $\mathcal{H}^{1}_{+}$ are preserved by $G\times \C^{*}$. In \cite{CW19}, where $(\dbar_{E},\Phi)$ is assumed to be stable, $S_+$ is referred to as the BB-\emph{slice}. In that case, the image of $S_+$ in $\mathcal{M}_{H}$ is the upward flow through $[(\dbar_{E},\Phi)]$. However, when $(\dbar_{E},\Phi)$ is only assumed to be polystable, $S_+$ almost never captures the full upward flow. In any case, $\kappa$ gives a biholomorphism between the $S_{+}$ and $\mathcal{H}^{1}_{+}$.

\begin{lemma}
The restriction $\kappa: S_{+}\to \mathcal{H}^{1}_{+}$ is a $G$- and $\C^{*}$-equivariant biholomorphism.
\end{lemma}
\begin{proof}
We claim that for $x=(\alpha,\upsilon)\in \mathcal{H}^{1}_{+}$, there exists a unique $(\beta,\varphi)\in S_{+}$ with $\kappa(\beta,\varphi)=x$. It amounts to solving
    \begin{equation}\label{equation_graded_kuranishi}
(\alpha,\upsilon)=(\sum_{j>0}\beta_j,\sum_{\ell\geq 0}\varphi_{\ell})+iD'\Lambda\Gamma\bigg(\sum_{j,\ell}[\beta_j,\varphi_{\ell}]\bigg).
    \end{equation}
As each $[\beta_j,\varphi_{\ell}]$ has graded degree $j+\ell$, we see that in particular $(\beta_1,\varphi_0)=(\alpha_1,\upsilon_0)$. Next, we solve inductively,
    \begin{equation}\label{equation_kuranishi_inverse}
(\beta_{j+1},\varphi_j)=(\alpha_{j+1},\upsilon_j)-(D'')^{*}\Gamma\bigg(\sum_{a+b=j} [\beta_a,\varphi_b]\bigg).
    \end{equation}
Putting these graded pieces together determine the solution $(\beta,\varphi)$ uniquely.

Lastly, we check that $(\beta,\varphi)\in S_{+}$. Since $D'x=0$, from \eqref{equation_graded_kuranishi} we have $D'(\beta,\varphi)=0$. Similarly, we have $0=D''x=D''(\beta,\varphi)+(1-P)([\beta,\varphi])=D''(\beta,\varphi)+[\beta,\varphi]-P([\beta,\varphi])$. Since $[\beta,\varphi]$ has no non-negative graded degree pieces, $P([\beta,\varphi])=0$ by Lemma \ref{lemma_h0_blockdiagonal}. Therefore, $(\beta,\varphi)\in S_+$. From \eqref{equation_kuranishi_inverse}, we see that the inverse map is holomorphic.
\end{proof}
Hence, $\kappa$ induces a $\C^{*}$-equivariant analytic isomorphism
    \begin{equation}
\overline{\kappa}: S_{+}\git G\xrightarrow{\cong} \mathcal{H}^{1}_{+}\git G.
    \end{equation}
Note that $\mathcal{H}^{1}_{+}\git G$ here is an affine GIT quotient. At the pre-quotient level, the resulting map of the following composition $$S_{+}\to (\dbar_{E},\Phi)+S_{+}\hookrightarrow \mathcal{B}\to \mathcal{B}/\mathcal{G}_{\C}$$
is $\C^{*}$-equivariant. As $(\dbar_{E},\Phi)+S_{+}\cap\tilde{U}$ consists of semistable Higgs bundles for small open $\tilde{U}$ containing $0$, the $\C^{*}$-equivariance implies that we have the global map $$S_{+}\longrightarrow \mathcal{M}_{H}.$$
Composing with $\kappa^{-1}$, we also have the map $\mathcal{H}^{1}_{+}\to \mathcal{M}_{H}$. When $(\dbar_{E},\Phi)$ is stable, it is shown in \cite{CW19} that this gives a parametrization of the upward flow through $[(\dbar_{E},\Phi)]$. This is not the case if $(\dbar_{E},\Phi)$ is strictly polystable. Nonetheless, using $\C^{*}$-equivariance again and the local complex analyticity of the map $(S_{+}\cap G\cdot \tilde{U})\git G\hookrightarrow \mathcal{M}_{H}$, we obtain a global injective $\C^{*}$-equivariant complex analytic map $L:S_+\git G\to \mathcal{M}_{H}$ given by
    \begin{equation}
        L([(\beta,\varphi)])=[(\dbar_u,\Phi_u)].
    \end{equation}
We call the image of $L$ the \emph{central locus} of the upward flow through $[(\dbar_{E},\Phi)]$. Composing with $\overline{\kappa}^{-1}$, we then obtain

\begin{proposition}\label{proposition_central_locus}
The map $L\circ \overline{\kappa}$ gives a parametrization of the central locus of the upward flow through the polystable $\C^{*}$-fixed point $[(\dbar_{E},\Phi)]$ by the affine quotient $\mathcal{H}^{1}_{+}\git G$.
\end{proposition}

For further discussion about the central locus, it will be of interest to us to consider certain subloci in $S_{+}$. To define them, we first set up some notation and terminology. Let $\{(\mathcal{E}^{\alpha},\phi^{\alpha})\}$ be the maximal set of pairwise non-equivalent stable Higgs bundles, each gauge equivalent to some $(\dbar^{i},\Phi^{i})$ in \eqref{equation_stable_summand_decomposition}. Let $C$ be the collection containing $n_{\alpha}$ copies of $(\mathcal{E}^{\alpha},\phi^{\alpha})$, where $n_{\alpha}$ is the number of copies of $(\mathcal{E}^{\alpha},\phi^{\alpha})$ in $(\dbar_{E},\Phi)$. We do not distinguish members in $C$ which are isomorphic to each other.

Define a \emph{partition} (of $C$) to be a collection $P=\{P_{\lambda}\}$ of subcollections of $C$ which partitions $C$. Let $P^C=\{C\}$ be the trivial partition. We say $P$ is \emph{proper} if $P\neq P^C$. Each $P_{\lambda}\in P$ defines a polystable Higgs bundle $(\dbar(P_{\lambda}),\Phi(P_{\lambda}))$ obtained from taking the direct sum of members in $P_{\lambda}$. Note that $(\dbar(P_{\lambda}),\Phi(P_{\lambda}))$ represents a $\C^{*}$-fixed point. Now, for each partition $P$, we fix an isomorphism between each factor in each $P_{\lambda}$ and a stable factor of $(\dbar_{E},\Phi)$ so that the resulting holomorphism $\bigoplus_{\lambda}(\dbar(P_\lambda),\Phi(P_{\lambda}))\to (\dbar_{E},\Phi)$ is an isomorphism of Higgs bundles. Endow each $(\dbar(P_{\lambda}),\Phi(P_{\lambda}))$ with the SHB structure laid out in \eqref{equation_stable_summand_decomposition}, and fix a harmonic metric that is compatible with that of $(\dbar_{E},\Phi)$. For each $(\dbar(P_{\lambda}),\Phi(P_{\lambda}))$, the induced $\C^{*}$-action allows one to define the model spaces, denoted by $S_{+}(P_{\lambda})$ and $\mathcal{H}^{1}_{+}(P_{\lambda})\subset \mathcal{H}^{1}(P_{\lambda})$ respectively, as well as the biholomorphism coming from the restricted Kuranishi map $$\kappa(P_{\lambda}): S_{+}(P_{\lambda})\to \mathcal{H}^{1}_{+}(P_{\lambda}).$$ 
Let $S_{+}(P):=\prod_{\lambda} S_{+}(P_{\lambda})$ and $\mathcal{H}^{1}_{+}(P):=\prod_{\lambda} \mathcal{H}^{1}_{+}(P_{\lambda})$. Notice that $S_{+}(P^{C})\cong S_{+}$ and $\mathcal{H}^{1}_{+}(P^{C})\cong \mathcal{H}^{1}_{+}$.

We introduce a partial ordering on partitions as follows. For two distinct partitions $P=\{P_{\lambda}\}$ and $P'=\{P'_{\lambda'}\}$, we write $P>P'$ if we can group members in $P'$ into a disjoint union $P'=\coprod_{\lambda}Q_{\lambda}$ such that for each $\lambda$, the union of members in $Q_{\lambda}$ is $P_{\lambda}$. Given a pair $P>P'$ and a grouping, we have the following $\C^{*}$-equivariant commutative diagram
     \begin{equation}\label{equation_partition_space}
    \begin{tikzcd}
    S_{+}(P') \arrow{r}{\prod_{\lambda'}\kappa(P'_{\lambda'})} \arrow{d} & [3em] \mathcal{H}^{1}_{+}(P') \arrow{d} \arrow{r} & \mathcal{H}^{1}(P') \arrow{d}\\
    S_{+}(P) \arrow{r}{\prod_{\lambda}\kappa(P_{\lambda})} & \mathcal{H}^{1}_{+}(P) \arrow{r} & \mathcal{H}^{1}(P)
    \end{tikzcd}
    \end{equation}
where the vertical arrows are closed embeddings.

\begin{remark}
In particular, since $P^{C}>P$ for any proper partition $P$, we may regard $S_{+}(P)$, $\mathcal{H}^{1}_{+}(P)$ and $\mathcal{H}^{1}(P)$ as closed subspaces of $S_{+}$, $\mathcal{H}^{1}_{+}$ and $\mathcal{H}^{1}$ respectively. In general, as subspaces of $S_{+}$, $\mathcal{H}^{1}_{+}$ and $\mathcal{H}^{1}$ respectively, $S_{+}(P)$, $\mathcal{H}^{1}_{+}(P)$ and $\mathcal{H}^{1}(P)$ are not preserved by $G$ unless $G$ is abelian. When $G$ is abelian and $P=P^{C}$, \eqref{equation_partition_space} is $G$-equivariant in every arrow.
\end{remark}

It is reasonable to expect that the (inverse) Kuranishi map $\theta$ \eqref{equation_kuranishi_inverse_theta} should give a correspondence between $G$-(poly)stability in $\mathcal{H}^{1}$ and (poly)stability of the image Higgs bundles. Indeed, it has been proved to be the case for holomorphic vector bundles in \cite{BS20}. While this does not appear to follow immediately from the construction in \cite{Fan22}, for our purposes it suffices to have a partial correspondence along $\mathcal{H}^{1}_{+}$ as well as imposing the abelian assumption on $G$. The precise formulation is given in the following

\begin{proposition}\label{proposition_stable_vector}
Let $u=(\beta,\varphi)\in S_{+}(\dbar_{E},\Phi)-\{0\}$, and assume $(\dbar_{u},\Phi_u)$ is not gauge equivalent to $(\dbar_{E},\Phi)$.
    \begin{enumerate}[label=(\roman*)]
        \item
Suppose $(\dbar_{u},\Phi_{u})$ is stable. Then $\kappa (u)$ is stable in $\mathcal{H}^{1}_{+}$.
        \item
Suppose $(\dbar_{u},\Phi_{u})$ is strictly polystable and the automorphism group $G$ of $(\dbar_{E},\Phi)$ is abelian. Then $\overline{G\cdot \kappa(u)}$ contains a polystable vector $y$ which lies in $\mathcal{H}^{1}_{+}(P)$ for some proper partition $P$ of $C$.
    \end{enumerate}
\end{proposition}
\begin{proof}
Assume that $(\dbar_{u},\Phi_{u})$ is stable. Using the $\C^{*}$-action, we may assume that $x=\kappa (u)\in U$ where $U$ is a sufficiently small ball for which Proposition \ref{proposition_yue_local_model} applies. Take a point $y\in \overline{G\cdot x}$ that is closest to the origin among points in $\overline{G\cdot x}$. In particular, $y\in U$ and $G\cdot y$ is closed. It follows from Proposition \ref{proposition_yue_local_model} that $\theta(y)$ is polystable and lies in the same gauge equivalence class as $(\dbar_{u},\Phi_u)$. Hence, $\theta(y)$ is also stable. It implies that its isotropy group is trivial, and so is its restriction to $G$. As the Kuranishi map $\kappa$ is a $G$-equivariant biholomorphism, $y$ is stable. Since stability is an open condition, $x$ is also stable and $y\in G\cdot x$. This proves (i).

Next, assume $(\dbar_{u},\Phi_{u})$ is strictly polystable and $G$ is abelian. It follows that $\kappa(u)$ is semistable and that the stable factors in \eqref{equation_stable_summand_decomposition} are pairwise non-equivalent. Let $\bigoplus_{i=1}^{k}(\mathcal{E}^{i}_{u},\dbar^{i}_{u},\Phi^{i}_{u})=(\dbar_{u},\Phi_{u})$ be its stable summand decomposition. From Lemma \ref{lemma_limit_direct_sum_commute}, we see that the downward limits of $(\mathcal{E}^{i}_{u},\dbar^{i}_{u},\Phi^{i}_{u})$ define a partition $$P=\{P_{i}:i=1,\cdots,k\},$$
where each $P_{i}$ contains stable summands of the downward limit of $[(\mathcal{E}^{i}_{u},\dbar^{i}_{u},\Phi^{i}_{u})]$. Let $U_i$ and $U$ be sufficiently small open balls around $0$ in $\mathcal{H}^{1}(P_i)$ and $\mathcal{H}^{1}$ respectively in the construction of the local Kuranishi model in Proposition \ref{proposition_yue_local_model}, such that $\prod_{i}U_i\subset U\subset \mathcal{H}^{1}$. Here, the inclusion comes from the remark after \eqref{equation_partition_space}. Without loss of generality, we can assume $[(\mathcal{E}^{i}_{u},\dbar^{i}_{u},\Phi^{i}_{u})]$ has a representative $x_i\in U_i$. Writing $x_P = \sum_i x_i$ and $u_P=\kappa^{-1}(x_P)$, we have $[(\dbar_{u_P},\Phi_{u_P})]=[(\dbar_u,\Phi_u)]$. Thus, $\overline{G\cdot \kappa(u)}\cap \overline{G\cdot x_P}\neq \emptyset$. Pick a point $y$ that lies in the unique closed $G$-orbit of this intersection. As both $\mathcal{H}^{1}(P)$ and $\mathcal{H}^{1}_{+}$ are closed and $G$-invariant, $y\in \mathcal{H}^{1}_{+}(P)$.
\end{proof}

It was shown in \cite{CW19} that the upward flow through a stable $\C^{*}$-fixed point is a complex Lagrangian with respect to the holomorphic structure $I$ of $\mathcal{M}^{s}_{H}$. The following proposition, assuming that the central locus intersects with the stable locus $\mathcal{M}^{s}_{H}$, can be viewed as an analogue to this statement.

\begin{proposition}\label{proposition_central_locus_lagrangian}
Let $u=(\beta,\varphi)\in S_{+}$ such that $(\dbar_u,\Phi_u)$ is stable. Then, $L$ is locally an immersion in a neighborhood of $[(\beta,\varphi)]\in S_+ \git G$ and is Lagrangian. 
\end{proposition}
\begin{proof}
Without loss of generality we assume that $\kappa(u)\in \mathcal{H}^{1}_{+}\cap U$, where $U$ is a sufficeintly small ball in the construction of the local Kuranishi model in Proposition \ref{proposition_yue_local_model}. First, we show that $L$ is an immersion at $[(\beta,\varphi)]$. Let $G\cdot (\mathcal{Z}\cap U)^{s}$ be the stable locus of $G\cdot (\mathcal{Z}\cap U)$. As $(\dbar_u, \Phi_u)$ is stable and $\mathcal{M}^{s}_{H}$ is open in $\mathcal{M}_{H}$, by the proof of (i) in Proposition \ref{proposition_stable_vector}, there exists an open neighborhood $V\subset U$ of $\kappa(u)$ where $\theta(x)$ (cf. \eqref{equation_kuranishi_inverse_theta} is stable for $x\in \mathcal{Z}\cap V$. Let $G\cdot (\mathcal{Z}\cap U)^{s}$ be the stable locus of $G\cdot (\mathcal{Z}\cap U)$. We have the $G$-invariant open set $G\cdot (\mathcal{Z}\cap V)\subset G\cdot (\mathcal{Z}\cap U)^{s}$ where $G$ acts freely and properly. Restricted to $\mathcal{H}^{1}_{+}$, we have the closed embedding $\mathcal{H}^{1}_{+}\cap G\cdot V\hookrightarrow G(\mathcal{Z}\cap V)$. At $\kappa(u)$, this gives the inclusions $$\text{Lie}(G)\hookrightarrow T_{\kappa(u)}\mathcal{H}^{1}_{+}\hookrightarrow T_{\kappa(u)}\mathcal{Z}.$$
Quotient by $\text{Lie}(G)$, we have $T_{[(\beta,\varphi)]}S_{+}\git G\hookrightarrow T_{[(\dbar_u,\Phi_u)]}\mathcal{M}^{s}_{H}$. This shows $L$ is locally an immersion around $[(\beta,\varphi)]$.

Let $\overline{V}:=(\mathcal{H}^{1}_{+}\cap G\cdot V)/ G$. Now, take a slice $s_u$ through $\kappa(u)$. Through $\kappa^{-1}|_{s_u}$, tangent vectors of the image of $L\circ \overline{\kappa}(\overline{V})$ are represented by elements in $C^{1}_{+}\subset \Omega^{0,1}(\sllie(E))\oplus \Omega^{1,0}(\sllie(E))$. It follows that $\Omega|_{L\circ\overline{\kappa}(\overline{V})}\equiv 0$. As $\dim \overline{V}=\dim\mathcal{H}^{1}_{+}-\dim G=\dfrac{1}{2}\mathcal{M}_{H}$ by Proposition \ref{proposition_dim}, we see that $L$ is Lagrangian around $\overline{\kappa}^{-1}(\overline{V})$.
\end{proof}

Finally, we show that under Assumption \ref{assumption_I} or Assumption \ref{assumption_II}, the central locus always intersect with $\mathcal{M}^{s}_{H}$.

\begin{theorem}\label{theorem_non_empty}
Suppose $\Phi=0$, or the automorphism group of $(\dbar_{E},\Phi)$ is abelian. Then $(\dbar_E,\Phi)+S_+$ contains a stable Higgs bundle.
\end{theorem}
\begin{proof}
For the first case $\Phi=0$, the underlying holomorphic vector bundle is polystable. Fix a harmonic metric $h$. Let $(\mathcal{H}^{1})^{0,1}:=\{\beta\in\Omega^{0,1}(\sllie(E)): \partial\beta=0\}$. We have the decomposition $\mathcal{H}^{1}=(\mathcal{H}^{1})^{0,1}\oplus \mathcal{H}^{1}_{+}$, where each factor is $G$-invariant. Taking $h$-hermitian adjoint gives the norm-preserving identification $(\mathcal{H}^{1})^{0,1}\xrightarrow{\cong_{\R}}\mathcal{H}^{1}_{+}$. Since $G^{*_{h}}=G$, this identification maps $G$-orbits onto $G$-orbits. Now, there exists $\beta\in (\mathcal{H}^{1})^{0,1}$ with small $h$-norm such that $\theta(\beta)=(\dbar_{E}+\beta,0)$ is a stable vector bundle. By the proof of (i) of Proposition \ref{proposition_stable_vector}, $\beta$ is $G$-stable. Kempf-Ness theorem then implies that $\varphi:=\beta^{*}$ is also $G$-stable. In particular, its isotropy group has dimension zero. By Proposition \ref{proposition_yue_local_model}, the resulting Higgs bundle $(\dbar_{E},\varphi)$ is polystable. But the automorphism group of $(\dbar_{E},\varphi)$ is a subgroup of $G$, and hence coincides with the isotropy group of $\varphi\in \mathcal{H}^{1}_{+}$. Therefore, $(\dbar_{E},\varphi)$ is a stable Higgs bundle.

Next, suppose the automorphism group $G$ of $(\dbar_{E},\Phi)$ is abelian. We use the notations in \eqref{equation_stable_summand_decomposition} and \eqref{equation_hodge_summand_decomposition} respectively to denote the stable summands and Hodge summands respectively. First, we claim that there exists non-trivial $\varphi_{i}\in H^{0}(\text{Hom}(E^{i}_{\ell_i},E^{i+1}_{1})\otimes K_X)$ for each $i$ where here we let $E^{k+1}=E^{1}$. Using Riemann-Roch, we compute
    \begin{equation}
h^{1}((E^{i+1}_{1})^{\vee}E^{i}_{\ell_i})=h^{0}((E^{i+1}_{1})^{\vee}E^{i}_{\ell_i})-\deg ((E^{i+1}_{1})^{\vee}E^{i}_{\ell_i})+\text{rk}(E^{i+1}_{1})\text{rk}(E^{i}_{\ell_i})(g-1).
    \end{equation}
Recall that each $(\dbar^{i},\Phi^{i})$ is a stable SHB, so $\deg E^{i}_{\ell_i}\leq 0$ while $\deg E^{i}_{1}\geq 0$, where the equalities hold if and only if $\ell_i=1$. Hence, $\deg((E^{i+1}_{1})^{\vee}E^{i}_{\ell_i})\leq 0$ and therefore $h^{1}((E^{i+1}_{1})^{\vee}E^{i}_{\ell_i})>0$, which proves the claim. Now, we take $\varphi$ to be $\sum_{i=1}^{k}\varphi_{i}$. Write $$G=\{\xi=(\xi_1,\cdots,\xi_{k-1}):\prod_{i=1}^{k-1}\xi_{i}^{r_i}=1\}.$$
The $G$-action on $\varphi$ (writing $\xi_{k+1}=\xi_{1}$) can be written as $$\xi\cdot \varphi=\sum_{i=1}^{k}\xi_{i+1}\xi_{i}^{-1}\varphi_{i}.$$
Let $t=(t^{c_1},\cdots,t^{c_k})$ be a one-parameter subgroup of $G$, so $\sum_{i=1}^{k}r_i c_i=0$. The action of $t$ on $\varphi$ is given by $t\cdot \varphi=\sum_{i=1}^{k}t^{c_{i+1}-c_i}\varphi_{i}$. Notice that $c_{i+1}=c_{i}$ for all $i$ if and only if $c_i\equiv 0$. Moreover, $\sum_{i=1}^{k}(c_{i+1}-c_{i})=0$. It follows that either $t$ is trivial or there exists $i\neq j$ such that $c_{i+1}-c_i>0$ and $c_{j+1}-c_j<0$. By the Hilbert-Mumford criterion, $\varphi$ is $G$-stable, so $\mathcal{H}^{1}_{+}$ contains some $G$-stable point. In particular, together with Proposition \ref{proposition_dim}, we have $$\dim (\mathcal{H}^{1}_{+}\git G)=\dim \mathcal{H}^{1}_{+}-\dim G=\dfrac{1}{2}\dim \mathcal{M}_{H}=(r^2-1)(g-1).$$

Now, let $(\dbar_u,\Phi_u)\in (\dbar_{E},\Phi)+ S_{+}$. Replacing $u$ with some point in $\overline{G\cdot u}$ we may assume $\kappa(u)$ is $G$-polystable and $(\dbar_u,\Phi_u)$ is polystable. Also, using the $\C^{*}$-action, we may move $\kappa(u)$ into an open ball $U$ used in the construction of the local Kuranishi model in Proposition \ref{proposition_yue_local_model}. If $u\neq 0$ and $(\dbar_u,\Phi_u)$ is not stable, then by (ii) of Proposition \ref{proposition_stable_vector} and upon replacement if needed, $(\dbar_{u},\Phi_{u})\in (\dbar_{E},\Phi)+S_{+}(P)$ for some proper partition $P$. Hence, if $(\dbar_{E},\Phi)+S_{+}$ contains no stable Higgs bundles, then we can write $\mathcal{H}^{1}_{+}\git G=\bigcup_{P} \mathcal{H}^{1}_{+}(P)\git G$, a finite union of closed subsets ranging over proper partitions $P$.

Lastly, we proceed by induction on $\dim G$. First, we look at the case $\dim G=1$, in whise case $P=\{P_1,P_2\}$ is the unique proper partition and $G$ acts trivially on $\mathcal{H}^{1}_{+}(P)$. Let $r_1$ and $r_2$ be the rank of the underlying vector bundles of the two stable factors so that $r=r_1+r_2$. We have $$\dim (\mathcal{H}^{1}_{+}(P)\git G)=\dim\mathcal{H}^{1}_{+}(P_1)+\dim\mathcal{H}^{1}_{+}(P_2)=(r_1^{2}+r_{2}^{2}-2)(g-1)<\dim \mathcal{H}^{1}_{+}\git G.$$
Therefore, there must exist stable Higgs bundles on $(\dbar_{E},\Phi)+S_{+}$ since otherwise one would have $\mathcal{H}^{1}_{+}\git G=\mathcal{H}^{1}_{+}(P)\git G\cong \mathcal{H}^{1}_{+}(P)$.

Now, let $r>1$. We suppose it is true that the ($\dbar_{E},\Phi)+S_{+}$ contains a stable Higgs bundle whenever $\dim G<r$. Assume now that $\dim G=r\geq 2$. We claim that for each proper $P$, there exists polystable $w\in \mathcal{H}^{1}_{+}(P)$ such that  $w \notin\mathcal{H}^{1}_{+}(P')$ for all $P'\rlap{\kern.6em$|$}> P$. It suffices to show that for each subcollection of stable factors $P_{\lambda}\in P$, the locus $(\dbar(P_{\lambda}), \Phi(P_{\lambda}))+S_{+}(P_{\lambda})$ contains a stable Higgs bundle. If $|P_{\lambda}|=1$, $(\dbar(P_{\lambda}), \Phi(P_{\lambda}))$ is stable, in which case it is shown in \cite{CW19} that $(\dbar(P_{\lambda}), \Phi(P_{\lambda}))+S_{+}(P_{\lambda})$ consists of stable Higgs bundles. For $|P_{\lambda}|=r_\lambda+1>1$, we have $0<r_{\lambda}<r$ so this case follows from the induction hypothesis. Hence, if $(\dbar_{E}+\Phi)+S_{+}$ contains no stable Higgs bundles, then $\mathcal{H}^{1}_{+}\git G$ is a finite union of proper closed subsets, which contradicts to the fact that $\mathcal{H}^{1}_{+}$ is irreducible. Therefore, $(\dbar_{E},\Phi)+S_{+}$ must contain a stable Higgs bundle.
\end{proof}

%% file: proof.tex
\section{Conformal limits}

Let $(\dbar_E, \Phi)$ be a strictly polystable such that $[(\dbar_E, \Phi)]$ is a $\C^{*}$-fixed point. We will use the SHB structure introduced in Section \ref{subsection_SHB_deformation} and follow the notation therein. From the SHB structure, we obtain a family of isomorphisms from $(\dbar_{E},R\Phi)$ to $(\dbar_{E},\Phi)$ for each $R>0$. For a harmonic metric $h$ of $(\dbar_{E},\Phi)$, the summands $E_a$ and $E_b$ are mutually orthogonal. By pulling back $h$, we obtain harmonic metrics $h_R$ for $(\dbar_{E},R \Phi)$.

Let $(\dbar_{u},\Phi_u)\in S_+$ be stable. The idea in \cite{DFK+21} of establishing the existence of the conformal limit of $(\dbar_{u},\Phi_u)$ is the following. Using $h_R$, one first obtain a family of connections $$\dbar_{u}+\partial^{h_R}_{u}+\hbar^{-1} \Phi_u+\hbar R^{2}\Phi_{u}^{*_{h_R}}$$
which converges to the flat connection
    \begin{equation}
D_0:= \dbar_{u}+\partial^{h}_{E}+h^{-1}\Phi_{u}+\hbar \Phi^{*_h}.        
    \end{equation}
Then, one attempt to correct the metrics to $(e^{f_R})^{*}h_R$, where $f_R$ is $h_R$-hermitian, so that the resulting connections are flat. This amounts to solving
    \begin{equation}\label{equation_M_functional}
N_{u,R}(f_R):= \dbar_u(e^{-2f_R}\partial_u^{h_R}(e^{2f_R}))+F_{(\dbar_u,h_R)}+[\Phi_u,e^{-2f_R}(R^2 \Phi_{u}^{*_{h_R}})e^{2f_R}] = 0.
    \end{equation}
Recall that the SHB structure induces a grading on $\Omega^{p,q}(\text{End}(E))$. Using the notation in Section \ref{subsection_SHB_deformation}, for $d>0$, we have  $$(f_R)_{-d}=R^{2d}(f_{R})_d^{*_h}.$$
This fact is used in \cite{CW19} to parametrize $f_R$ by non-negatively graded pieces so that the operator $M_{u,R}$ is extended to all $R$. When the SHB $(\dbar_{E},\Phi)$ is stable, using the implicit function theorem, it is shown that for small $R$ there exists $f_R$ with $f_0=0$ such that $M_{u,R}(f_R)\equiv 0$, which confirms that the conformal limit of $(\dbar_u,\Phi_u)$ is $D_0$

In the present context, the last perturbation step in the method outlined above breaks down due to the non-uniqueness harmonic metrics of $(\dbar_{E},\Phi)$. In the following, we will see that the issues that arise from this can be overcome under Assumption \ref{assumption_I} or Assumption \ref{assumption_II}. In both cases, we will determine a harmonic metric $h$ which depends on $u$, from which we determine a family of metrics $h'_R$ modified from $h_R$. The first case is technically simpler due to the absence of non-trivial SHB structures and will be addressed in the next subsection. On the other hand, more preparations are needed to treat the second case. Combining Theorem \ref{theorem_conformal_limit_open_stratum} and Corollary \ref{corollary_conformal_abelian}, we then obtain Theorem \ref{theorem_CL}.

\subsection{On the open stratum}\label{section_open_stratum}

Suppose $\Phi=0$, so that $(E,\dbar_{E})$ is a polystable vector bundle. In this case, $$S_{+}=\mathcal{H}^{1}_{+}=H^{0}(\sllie(E)\otimes K_{X})$$ and the Kuranishi map $\kappa$ restricted to $S_{+}$ is the identity (cf. \eqref{equation_S+}). For a harmonic metric $h$ of $(\dbar_{E},0)$, we denote the subgroup of the automorphism group $G$ of $(\dbar_{E},0)$ that fixes $h$ by $K(h)$ or simply by $K$ whenever no confusions arise. It is a maximal compact subgroup of $G$ whose Lie algebra is given by $iH_0(h)$, where
    \begin{equation}
H_0(h):=\{v\in \mathcal{H}^{0}: v^{*_{h}}=v\}.
    \end{equation}
As $G$ acts transitively by pulling back on the set of harmonic metrics of $(\dbar_{E},0)$, by Cartan decomposition we may parametrize the set of harmonic metrics by $H_0(h)$. Under Assumption \ref{assumption_I}, we have $(\dbar_u,\Phi_u)=(\dbar_{E},\varphi)$.

Start with an arbitrary harmonic metric $h_0$ of $(\dbar_{E},0)$. Consider the $G$-representation $V=\mathcal{H}^{1}_{+}$, equipped with an inner product induced from $h_0$ so that $K(h_0)$ acts unitarily. Recall that we assume $(\dbar_u,\Phi_u)$ is stable. By Proposition \ref{proposition_stable_vector}, $\varphi$ is a stable vector in $V$. Hence, there exists $g\in G$ such that the norm of $g\cdot \varphi$ is minimized along the $G$-orbit of $\varphi$. Let $h=g^{*}h_0$, i.e. $h(v,w)=h_0(gv,gw)$. To make the notations more concise, below we shall adopt the abbreviations $H_0=H_0(h)$ and we denote the $h$-hermitian adjoint simply by $\varphi^{*}=\varphi^{*_h}$.

Define the Kempf-Ness functional $\ell_{\varphi}: H_0 \to \R$ by 
    \begin{equation}\label{equation_KN_functional}
        \ell_{\varphi}(g):= ||e^{g}\varphi e^{-g}||^{2}_{h} = i\int_{X}\tr \big[e^{g}\varphi e^{-g} (e^{g}\varphi e^{-g})^{*}\big]= i\int_{X}\tr \big[\varphi (e^{-2g}\varphi^{*}e^{2g})\big].
    \end{equation}
By construction and the Kempf-Ness theorem (\cite{KN79}), $\ell_{\varphi}$ attains its unique minimum at $0$. Hence, its Jacobian $D\ell_{\varphi}$ vanishes and the Hessian $H(\ell_{\varphi})$ is non-degenerate at the origin. For $g, v\in H_0$, we write
    \begin{equation}
v_{g}:= D_{g}\text{exp}^{2}(v)=\dfrac{d}{dt}\bigg|_{t=0}e^{2(g+tv)}\in \mathcal{H}^{0}.
    \end{equation}
The Jacobian of $\ell_{\varphi}$ is computed as follow:
    \begin{align}
-iD_{g}\ell_{\varphi}(v)&:= \dfrac{d}{dt}\bigg|_{t=0} \int_{X} \tr \big[\varphi(e^{2(g+tv)})^{-1}\varphi^{*}e^{2(g+tv)} \big] \\
&= \int_{X}\tr \big[\varphi (-e^{-2g}v_{g}e^{-2g})\varphi^{*}e^{2g}+\varphi e^{-2g}\varphi^{*}v_g \big] \nonumber\\
&= \int_{X} \tr \big(e^{-2g}\varphi^{*}e^{2g}\varphi+\varphi e^{-2g}\varphi^{*}e^{2g}\big) e^{-2g}v_g \nonumber\\
&= \int_{X} \tr [\varphi, e^{-2g}\varphi^{*}e^{2g}]e^{-2g}v_g \label{eqn_1st_order}.
    \end{align}

In particular, letting $g=0$, we have
    \begin{equation}\label{equation_jacobian}
-iD_0\ell_\varphi (v) = 2\int_{X}\tr [\varphi, \varphi^{*}]v=0
    \end{equation}

for any $v\in H_0$.

The second variation can be computed as follow.
    \begin{align}
-i\dfrac{\partial^{2}}{\partial w \partial v}\ell_{\varphi}(g) &:= \dfrac{d}{dt}\bigg|_{t=0} \int_{X} \tr \big[\varphi, (e^{2(g+tw)})^{-1}\varphi^{*} e^{2(g+tw)}\big](e^{2(g+tw)})^{-1}v_{g+tw} \\
&= \int_{X} \tr \bigg\{ \big[\varphi, [e^{-2g}\varphi^{*}e^{2g}, e^{-2g}w_g]\big]e^{-2g}v_{g} -\big[\varphi, e^{-2g}\varphi^{*}e^{2g}\big]\dfrac{d}{dt}\bigg|_{t=0}(e^{2(g+tw)})^{-1}v_{g+tw} \bigg\} \label{eqn_2nd_order}.
    \end{align}

Specializing at $0$, the second term vanishes because of (\ref{equation_jacobian}). Therefore, we obtain

\begin{lemma}\label{lemma_local_minima}
Let $(\dbar_{E},\varphi)$ be a stable Higgs bundle whose underlying holomorphic vector bundle is polystable. Then there exists a unique harmonic metric $h$ of $(\dbar_{E},0)$ such that
    \begin{enumerate}
        \item
$\int_{X} \tr[\varphi, \varphi^{*_h}]v = 0$ for every $v\in H_0$, and
        \item
for $w\in H_0$, if $\int_{X} \tr[\varphi, [\varphi^{*_h},w]]v \equiv 0$ for all $v\in H_0$, then $w=0$.
    \end{enumerate}
\end{lemma}

Henceforth, we fix the harmonic metric $h$ to be the one obtained in Lemma \ref{lemma_local_minima}. Before we proceed to the proof the existence of the conformal limit of $(\dbar_{E},\varphi)$, we first fix some notation and setting. Let
    \begin{equation}
p:\sllie(E)\to i\mathfrak{su}(E,h)        
    \end{equation}
be the bundle morphism which fibrewise projects a traceless endomorphism to its hermitian part. Let
$$H_1:= H_0^{\perp}\subset L^{2}_{k}(i\mathfrak{su}(E,h))$$
be the $L^{2}$ orthogonal complement in $L^{2}_{k}$-completion of $\Omega^{0}(i\mathfrak{su}(E,h))$ for some sufficiently large $k$. The two spaces of $\Omega^{0}(i\mathfrak{su}(E,h))$ and $\Omega^{2}(i\mathfrak{su}(E,h))$ are identified via the Hodge star. This identification extends to their $L^{2}_{k}$-completions and preserves the harmonic forms.

As $\Phi=0$, we will set $h_R=h$ for all $R>0$. We then obtain the following family of connections
    \begin{equation}
        \dbar_{E}+\partial_{E}^{h}+\hbar^{-1}\varphi+\hbar R^{2}\varphi^{*}.
    \end{equation}
The operator \eqref{equation_M_functional} is reduced to
    \begin{equation}
N(f,R) := \dbar_{E}(e^{-2f}(\partial^{h}_{E}e^{2f}))+R^2[\varphi, e^{-2f}\varphi^{*}e^{2f}].
    \end{equation}
First, we notice the following
    \begin{lemma}\label{lemma_open_stratum_1st_term_redundant}
For all $v\in H_0$, $\int_{X} \tr \ p(i\dbar_{E}(e^{-2f}\partial^{h}_{E}e^{2f}))v=0$.
    \end{lemma}
\begin{proof}
For any $A\in \Omega^{1}(\text{End}(E))$, we notice that
    \begin{align*}
2\int_{X}\tr \ p(\dbar_{E}A)v &= \int_{X}\tr \ (\dbar_{E}A+(\dbar_{E}A)^{*})v\\
&= \int_{X}\tr \ [\dbar_{E}(A v)+(\dbar_{E}(vA))^{*}]\\
&= \int_{X}\dbar_{E}\tr \ Av+\overline{\int_{X} \dbar_{E} \tr \ vA}\\
&= 0.
    \end{align*}
\end{proof}

We now arrive at the most important outcome of the above setup.

\begin{proposition}\label{proposition_open_stratum_perp_image}
There exists an open neighborhood $U$ of $0$ in $H_1$ and a smooth map $g: U\to H_0$ such that $g(0)=0$ and for any $v\in H_0$, $$\int_{X} \tr \ p\bigg(i[\varphi, e^{-2(f+g(f))}\varphi^{*}e^{2(f+g(f))}]\bigg)v =0.$$
\end{proposition}
\begin{proof}
Using the trace pairing, we regard the term $[\varphi, e^{-2f}\varphi^{*}e^{2f}]$ as a dual element of $H_0$. More precisely, define a map $n:H_1\oplus H_0\to H_0^{\vee}$ as follows: for $(f,g)\in H_1\oplus H_0$ and $v\in H_0$,
    \begin{equation}
n(f,g)(v) := \int_{X} \tr \ p\bigg(i[\varphi, e^{-2(f+g)}\varphi^{*}e^{2(f+g)}]\bigg)v.
    \end{equation}
Notice that $n(0,0)=0$ by Lemma \ref{lemma_local_minima}. Next, we compute the differential at the origin and along the $H_0$-direction: for $w\in H_0$,
    \begin{align*}
D_{(0,0)}n(w)(v) &= 2\int_{X} \tr \ p (i[\varphi,[\varphi^{*},w]])v \\
&= -2\text{Im}\int_{X} \tr \ [\varphi,[\varphi^{*},w]]v\\
&= -2\int_{X} \tr \ [\varphi,[\varphi^{*},w]]v.
    \end{align*}

By the second part of Lemma \ref{lemma_local_minima}, $D_{(0,0)}n$ is non-singular. Hence, by the implicit function theorem, we obtain a map $g:U\to H_0$, for some open neighborhood $U$ in $H_1$ containing $0$, such that $g(0)=0$ and $n(f,g(f))\equiv 0$.
\end{proof}

Define the map
$$N'(f,R):=p(iN(f+g(f),R)),$$
where $f\in H_1$ and $g: U\to H_0$ is the map obtained in the above proposition. The image of $N'$ lies in the Hodge dual of $H_1$. Next, we compute the differential of $N'$ at $(0,0)$ along the $H_1$-direction: $$dN'_{(0,0)}(f) = 2p \big(i\dbar_{E}\partial^{h}_{E}(f+D_0 g(f))\big)
= 2i\dbar_{E}\partial^{h}_{E}f$$
since $D_0 g(f)\in \mathcal{H}^{0}$ and $i\dbar_{E}\partial^{h}_{E} f$ is $h$-hermitian. As the Laplacian $i\dbar_{E}\partial^{h}_{E}$ has index zero with trivial kernel on $H_1$, by the implicit function theorem, we have a family $f_R\in H_1$ with $f_0=0$ for small $R$, such that $$p(iN(f_R+g(f_R),R))\equiv 0.$$
We notice that $iN(f_R+g(f_R),R)$ is $(e^{f_R+g(f_R)})^{*}h$-hermitian. To see that $N(f_R+g(f_R),R)\equiv 0$ for small $R$, we use the following

\begin{lemma}\label{lemma_open_stratum_small_projection}
Let $h_1$ and $h_2$ be two hermitian inner products on a finite dimensional vector space $V$ that induce the same volume form. Let $A$ be the $h_1$-hermitian endomorphism such that $h_2(\cdot,\cdot) = h_1(A\cdot,\cdot)$. Write $\text{Herm}_{0}(V,h_i)$ to be the space of traceless $h_i$-hermitian endomorphisms, $i=1,2$, and $$p_A:\text{Herm}_{0}(V,h_2)\to \text{Herm}_{0}(V,h_1)$$ the projection map with respect to the inner product on $\text{End}(V)$ induced from $h_1$.
If $||A||_{h_1}$ is sufficiently small, then $p_A$ is an $\R$-linear isomorphism.
\end{lemma}
\begin{proof}
Let $U$ be a neighborhood around $0$ of $\text{Herm}_{0}(V,h_1)$ such that if $B\in U$ then $h_B:=h_1(B\cdot,\cdot)$ is a hermitian inner product. Let $*_{h_B}$ be the conjugation with respect to the metric $h_B$. It follows that $\text{Herm}_{0}(V,h_B)$ is the fixed point set of $*_{h_B}$ and $\text{Herm}_0:=\{(f,B)\in \text{End}_{0}(V)\times U: f\in \text{Herm}_{0}(V,h_B)\}$ is a vector bundle over $U$.

Consider the bundle morphism $p_{\cdot}:\text{Herm}_0 \to \text{Herm}_{0}(V,h_1)$. The result then follows from continuity and that at $B=0$, $p$ is the identity.
\end{proof}

Now, if $k$ is sufficiently large, then $N(f_R+g(f_R),R)\in L^{2}_{k-2}(\mathfrak{su}(E,h)\otimes \bigwedge^{2}T^{*}X)\subset C^{0}(\mathfrak{su}(E,h)\otimes \bigwedge^{2}T^{*}X)$. By compactness of $X$ and Lemma \ref{lemma_open_stratum_small_projection}, $N(f_R+g(f_R),R)\equiv 0$ for sufficiently small $R$. This proves

\begin{theorem}\label{theorem_conformal_limit_open_stratum}
For a stable Higgs bundle $(\dbar_{E},\varphi)$ whose underlying holomorphic vector bundle $(E,\dbar_{E})$ is polystable, its $\hbar$-conformal limit is $\dbar_{E}+\partial^{h}_{E}+\hbar^{-1}\varphi$.   
\end{theorem}

\subsection{On polystable SHBs with pairwise distinct stable summands}

In this subsection, we assume the stabilizer of $(\dbar_{E},\Phi)$ is abelian, i.e. the stable summands $(\dbar^{i},\Phi^{i})$ in \eqref{equation_stable_summand_decomposition} are pairwise non-isomorphic. For $v\in \Omega^{p,q}(\text{End}(E))$, we write $v=\sum_{i,j,a,b} v_{jb|ia}$, where $v_{jb|ia}\in \Omega^{p,q}(\text{Hom}(E^{i}_{a},E^{j}_{b}))$. To write down an explicit description of the $G$-action on $S_{+}$, we first embed $G$ into $(\C^{*})^{k}$. Let $\chi_{i}$ be the standard character of $(\C^{*})^{k}$ obtained from projection to the $i$-th factor. Introduced in Section \ref{subsection_SHB_deformation}, $G\times \C^{*}$ acts on  $(\beta,\varphi)$ in the following way:
    \begin{equation}
        (t,s)\cdot (\beta,\varphi) = \bigg(\sum s^{b-a} t_{i} (t_{j})^{-1} \beta_{jb|ia}, \sum s^{b-a+1} t_{i} (t_{j})^{-1}\varphi_{jb|ia} \bigg).
    \end{equation}
Consider the finite dimensional vector space $V$ spanned by the components $\beta_{jb|ia}$ and $\varphi_{jb|ia}$. Note that $G\times \C^{*}$ acts on the whole $V$ linearly, and $V$ inherits a hermitian inner product from a harmonic metric $h$ of $(\dbar_{E},\Phi)$, such that $K(h)\times S^{1}\subset G\times \C^{*}$ acts unitarily. Since the $G$-orbit of $u$ is closed in $S_+$, it is also closed in $V$. Hence, $u$ is a $G$-stable vector in $V$. In the following, we will use the notation introduced in Section \ref{section_linear_rep}.

Next, we are going to introduce the following procedure:

Start with \begin{inparaenum}[(i)]
    \item $u_0=u=(\beta,\varphi)$,
    \item $G_0=G$,
    \item $c_{-1}=0\in \Q$,
    \item $x_{-1}=0\in X(G)^{\vee}\otimes_{\Z}\Q\subset i\mathfrak{k}$ and
    \item $S_{-1}=\emptyset$.
\end{inparaenum}

Suppose we have determined $u_{i+1}$, $G_{i+1}$, $c_i$ , $x_i$ and $S_i$ for $-1\leq i<n$. In the $n$-th iteration (starting with $n=0$):

Let $\nu_n$ be the sum of components of $u_n$ fixed by $G_n$. Let 
    \begin{equation}
\phi_{n}:=\sum_{i=0}^{n}\nu_i+\sum_{i=-1}^{n-1}P_{S_i}(u).        
    \end{equation}
Here, $P_{S_i}(u)$ refers to the projection of $u$ to $\bigoplus_{\lambda\in S_i} V_{\lambda}$.

Consider $\Lambda_{G\times \C^{*}}(u-\phi_n)\subset (i\mathfrak{k}_{0})^{\vee}\oplus \R$. The inclusion $G_{n}\to G$ induces a projection $$\pi_{n}: (i\mathfrak{k})^{\vee}\oplus \R\to (i\mathfrak{k}_n)^{\vee}\oplus \R,$$
where $\mathfrak{k}_{n}$ is the Lie algebra of the maximal compact of $G_{n}$. 
Define the transformation the following transformation on $i\mathfrak{k}\oplus \R$
    \begin{equation}
s_n(\ell,\rho):=(\ell,\rho+\langle \ell, \sum_{i=-1}^{n-1}x_i\rangle).        
    \end{equation}
Let $C_n$ be the convex hull of $\pi_{n}\circ s_n(\Lambda_{G\times \C^{*}}(u-\phi_n))$. Let
    \begin{equation}
\{0\}\times [c_n,c_n']:=C_n\cap \{0\}\times \{\rho: \rho>0\}        
    \end{equation}
the intersection of the positive $\rho$-ray with $C_{n}$. Then, verify that $c_{n-1} < c_n\in \Q$.

Let $F_n$ be the face of $C_n$ that contains the point $(0,c_n)$ in its interior (as a CW-complex), and
    \begin{equation}
S_n:= \Lambda_{G\times \C^{*}}(u)\cap s_{n}^{-1}(F_n).        
    \end{equation}
Verify that $P_{S_n}(u)$ is $G_n$-polystable. Replace the metric $h$ by some metric in $(G_{n})^{*}h$, with respect to which $P_{S_{n}}(u)$ attains the minimum norm along its $G_{n}$-orbit, which exists by the Kempf-Ness theorem.

Lastly, we find some $x_n\in X(G_n)^{\vee}\otimes_{\Z}\Q\subset i\mathfrak{k}$ such that the following hold: for all $(\ell,\rho)\in S_n$ and $(\ell',\rho')\in \Lambda_{G\times \C^{*}}(u-\phi_n)-S_n$, we have
    \begin{equation}\label{equation_face}
        \rho+\langle \ell, \sum_{i=-1}^{n}x_i\rangle = c_n
    \end{equation}
and
    \begin{equation}\label{equation_polytope_upper}
        \rho'+\langle \ell', \sum_{i=-1}^{n}x_i\rangle> c_n.
    \end{equation}

Now, let $G_{n+1}$ be the stablizer of $P_{S_{n}}(u)$ and $u_{n+1}:= u-\phi_{n}-P_{S_{n}}(u)$. If $\dim G_{n+1}=0$, we stop. Otherwise, we proceed to the ($n+1$)-iteration.

The following is the first main technical input for the main result of this section.

\begin{proposition}\label{proposition_iterative_algorithm}
Suppose $u=(\beta,\varphi)$ is stable. Then the iterative procedure described above is well-defined and terminates at the $(k-1)$-th iteration for some $k\geq 1$. Along the way, the following data is obtained:
    \begin{enumerate}[label=(\roman*)]
        \item
a harmonic metric $h$ of $(\dbar_{E},\Phi)$, which induces a hermitian inner product on $V$
        \item
$(x,\sigma) \in X(G\times \C^{*})^{\vee}$, where $\sigma\in \Z_{>0}$ determines a one-parameter subgroup of $\C^{*}$ of weight $\sigma$
        \item
a sequence of tori $G=G_0\supset G_1\supset \cdots \supset G_k = \{1\}$
        \item
$\nu_0, \cdots, \nu_{k-1}\in V$ with $\nu_i$ fixed by $G_i$ for each $i=0,\cdots k-1$
        \item
$S_0, \cdots, S_{k-1}\subset \Lambda_{G\times \C^{*}}(u)$, where each $P_{S_i}(u)$ is $G_i$-polystable and fixed by $G_{i+1}$, for $i=0,\cdots, k-1$, each having minimal $h$-norm along its $G_i$-orbit, and
        \item
a sequence of integers $0<d_0<\cdots<d_{k-1}\in \Z_{>0}$, such that under the action of the one-parameter subgroup $(x,\sigma)$, we have
    \begin{equation}\label{equation_rescaled_R}
(x(t),\sigma(t)) \cdot u =  \sum_{i=0}^{k-1}(x(t),\sigma(t))\cdot \nu_i+\sum_{i=0}^{k-1} t^{d_i} P_{S_{i}}(u)+(x(t),\sigma(t))\cdot u_{k}.
    \end{equation}
    \end{enumerate}
Each $(x(t),\sigma(t))\cdot \nu_i$ can be further decomposed into
    \begin{equation}\label{equation_rescaled_R_2}
(x(t),\sigma(t))\cdot \nu_i = \sum_{j} t^{d_{ij}}\nu_{ij},         
    \end{equation}
where all $d_{ij}$'s are positive integers and $\min_{j} d_{ij} > d_{i-1}$.
Likewise, the residual term $(x(t),\sigma(t))\cdot u_{k}$ can be decomposed into 
    \begin{equation}\label{equation_rescaled_R_3}
(x(t),\sigma(t))\cdot u_{k} = \sum_{j} t^{e_j} (u_k)_{j},         
    \end{equation}
where $e_j$'s are positive integers and $\min_{j} e_{j}>d_{k-1}$.
\end{proposition}
\begin{proof}
Assume we have performed the procedure described above successfully from the $0$-th up to the $(n-1)$-th iteration. We are going to show that the $n$-th iteration can be carried out, giving \begin{inparaenum}[(i)]
    \item $\nu_n$,
    \item $c_n$,
    \item an updated metric $h$,
    \item a subset of weights $S_n$, and
    \item a solution $x_n$ to \eqref{equation_face} and \eqref{equation_polytope_upper},
\end{inparaenum}
all of which satisfy the required properties.

First, note that $\phi_{n}$ is fixed by $G_{n}$. As $u$ is $G$-stable, $u-\phi_{n}$ is $G_{n}$-stable by the Hilbert-Mumford criterion. Let $$q_{n}: (i\mathfrak{k}_{n})^{\vee}\oplus \R\to (i\mathfrak{k}_{n})^{\vee}$$ be the projection map. By Proposition \ref{prop-convexhull}, the convex hull of $q_{n}\circ\pi_{n}\circ s_{n}(\Lambda_{G\times \C^{*}}(u-\phi_n))$, which is $q_{n}(C_{n})$, contains $0$ in its interior. Hence, $C_{n}\cap \{0\}\times\{\rho: \rho>0\}$ is non-empty. As the $\rho$-coordinates of all points in $s_n(\Lambda_{G\times \C^{*}}(u-\phi_n))$ are greater than $c_{n-1}$, we have $c_{n-1}<c_{n}\in \Q_{>0}$ ($(0,c_{n})$ is a convex combination of these points).

Now, following the description of iterative procedure, the face $F_{n}$ contains $(0,c_n)$ in its interior. Note that $F_{n}$ is not a vertex, since otherwise $F_{n}=\{(0,c_n)\}$ but we have already removed $\nu_n$ from $u-\phi_{n}$. Next, we claim that $P_{S_{n}}(u)$ is $G_{n}$-polystable. Take a complementary torus of $G_{n+1}$ inside $G_{n}$ and let $\mathfrak{k}_{n}'$ be the Lie algebra of its maximal compact. So we have $(i\mathfrak{k}_{n})^{\vee}=(i\mathfrak{k}_{n+1})^{\vee}\oplus (i\mathfrak{k}_{n}')^{\vee}$. Let $\pi_{n}':(i\mathfrak{k}_{n})^{\vee}\to (i\mathfrak{k}_{n}')^{\vee}$ be the projection map. It follows that the convex hull of $\pi_{n}'(\Lambda_{G_{n}}(P_{S_n}(u)))$ contains $0$ in its interior. By Proposition \ref{prop-convexhull}, $P_{S_n}(u)$ is then $G_{n}'$-stable. In particular, its $G_{n}'$-orbit is closed. Since the $G_{n}$-orbit of $P_{S_n}(u)$ coincides with the $G_{n}'$-orbit, it follows that $P_{S_n}(u)$ is $G_{n}$-polystable. We replace the metric $h$ by pullback from some element $g\in G_{n}$ such that with respect to this metric $g^{*}h$, $P_{S_n}(u)$ attains the minimal norm along its $G_{n}$-orbit. By abuse of notation, we still denote the updated metric by $h$.

Next, we show that solutions to \eqref{equation_face} and \eqref{equation_polytope_upper} exists. Recall that  $\{0\}\times [c_n,c_n']:=C_n\cap \{0\}\times \{\rho: \rho>0\}$. We note that $\dim q_n(C_n)\leq \dim C_n \leq \dim q_n(C_n) +1$. Consider the possibilty $\dim q_n(C_n)=\dim C_n$ first. In this case, $C_n=F_n$ is contained in a hyperplane defined by $\rho t+\langle \ell,x\rangle = c_n t$ for some $(x,t)\in X(G_{n}\times \C^{*})^{\vee}\otimes_{\Z}\Q-\{0\}$. Observe that $t\neq 0$, since otherwise, we have $\langle \ell, x\rangle=0$ for all $\ell\in C_{n}$, which contradicts $G_n$-stability of $u-\phi_n$. As $t\neq 0$, we obtain the unique normalized solution $(x_n,1)$, where this $x_n$ solves \eqref{equation_face} and \eqref{equation_polytope_upper}. By stability, $\dim G_{n+1}=0$, so we stop the iteration here. In this case, $u_{n+1}=u-\phi_n - P_{S_n}(u)=0$.

Suppose $\dim C_n= \dim q_n (C_n)+1$. We first consider the `homogenized' system of inequalities
    \begin{equation}
        \bigg(\rho+ \langle \ell,\sum_{i=-1}^{n-1}x_i\rangle\bigg)t+\langle \ell, x_n \rangle = c_n t
    \end{equation}
and
    \begin{equation}\label{equation_homo_polytope_upper}
        \bigg(\rho'+\langle \ell',\sum_{i=-1}^{n-1}x_i\rangle\bigg)t+\langle \ell', x_n\rangle > c_n t.
    \end{equation}
for $(\ell,\rho)\in F_{n}$ and $(\ell,\rho')\in C_{n}-F_{n}$. By the condition that $F_{n}$ is a face and all vertices of $C_{n}$ have rational coordinates, this system always has rational solutions $(x_{n},t)$. For all such $(x_n, t)$, $t\neq 0$. Otherwise, we have $\langle \ell'',x_{n}\rangle \geq 0$ for all $\ell''\in q_n \circ\pi_{n}(C_{n})$, but this contradicts $G_n$-stability of $u-\phi_n$. Moreover, as $(0,c_n')$ also satisfies \eqref{equation_homo_polytope_upper}, we have $c_n't>c_n t$. Hence, we can find a normalized solution $(x_n,1)$, where $x_n$ solves \eqref{equation_face} and \eqref{equation_polytope_upper}, once we obtain the following
\begin{lemma}
Suppose $\dim C_n= \dim q_n (C_n)+1$, then $c_n<c_n'$.
\end{lemma}
\begin{proof}
Note that if $(\ell,\rho)$ is an interior point of $C_{n}$, then the small variations $(\ell,\rho+\epsilon)$ stay within the interior of $C_{n}$ when $|\epsilon|$ is sufficiently small. For $\dim C_{n}\geq 3$, take an interior point $(\ell,\rho)$ of $C_{n}$, find a hyperplane containing $\ell$ and $0$ in $(i\mathfrak{k}_{n})^{*}$, say defined by $\ker x$ where $x\in i\mathfrak{k}_{n}-\{0\}$. It suffices to prove the statement on the restriction $\{(\ell,\rho): \langle \ell,x\rangle =0\}\cap C_{n}$. Inductively, we can reduce to the case where $\dim C_{n}=2$ and $\dim q_{n}(C_{n})=1$. In this case, the vertical line $\{(\ell,\rho): \ell=0\}$ separates $C_{n}$ into two portions. Both of them must contain interior points of $C_{n}$, since otherwise the closure of one of them is the whole $C_{n}$, leaving the other portion empty. Now, pick some interior points $(\ell_1, \rho_1)$ and $(\ell_2,\rho_2)$ from each portion. There exists $\epsilon>0$ such that the slight variation $(\ell_1,\rho_1+ t)$ and $(\ell_2,\rho_2+ t)$ remain in the same portion for $t\in (-\epsilon,\epsilon)$. Now, for $0\leq \lambda\leq 1$, the segments $(\lambda \ell_1+(1-\lambda)\ell_2,\lambda \rho_1+(1-\lambda) \rho_2+t)$ must cross the vertical line. Above the spot where $\lambda \ell_1+(1-\lambda)\ell_2=0$, we see a variation $\lambda \rho_1+(1-\lambda) \rho_2+t$ with $|t|<\epsilon$, so we must have $c_{n}<c_{n}'$.
\end{proof}
Therefore, we can find such $x_n$. The $n$-th iteration is thus completed. For the stabilizer $G_{n+1}$ of $P_{S_n}(u)$, if $\dim G_{n+1}\neq 0$, we let $u_{n+1}=u-\phi_n-P_{S_n}(u)$ and proceed to the $(n+1)$-iteration. Else, we stop. Since $P_{S_n}(u)$ is not fixed by $G_n$, $\dim G_{n+1}<\dim G_{n}$, so the process will stop at the $(k-1)$-iteration for some $k\geq 1$.

By the above construction, we have $$u = \sum_{i=0}^{k-1}\nu_i+\sum_{i=0}^{k-1}P_{S_{i}}(u)+u_{k}.$$
We let $x=\sigma \sum_{i=0}^{k-1} x_i$ for some $\sigma\in \Z_{>0}$ so that $X(G)^{\vee}$. Regarding $(x,\sigma)$ as a one-parameter subgroup, its action on $u$ gives $$(x(t),\sigma(t))\cdot u = \sum_{i=0}^{k-1}(x(t),\sigma(t))\nu_i +\sum_{i=0}^{k-1} t^{c_i\sigma} P_{S_i}(u) + (x,t)\cdot u_{k}.$$
As in the statement of the lemma, write $(x(t),\sigma(t))\cdot \nu_i = \sum_{j} t^{d_{ij}\sigma}\nu_{ij}$ and $(x(t),\sigma(t))\cdot = \sum_{j} t^{e_j\sigma} (u_k)_{j}$. By construction, we have the bounds $\min_{j} d_{ij} > c_{i-1}$ and $\min_{j} e_{j}>c_{k-1}$. Now, pick $\sigma$ so that all $c_i\sigma$, $d_{ij}\sigma$ and $e_{j}\sigma$ are integers. Finally, we let $d_i$ to be $c_i \sigma$.
   \end{proof}

With $(x,t)$ and $h$ determined in Proposition \ref{proposition_iterative_algorithm}, as a one-parameter subgroup of gauge transformations, we obtain a family of metrics $h'_R$ parametrized by $R>0$,
    \begin{equation}
h'_R := (x(R),\sigma(R))^{*}h.
    \end{equation}
For $v\in \Omega^{p,q}(\text{End}(E))$, we have
    \begin{equation}\label{equation_adjoint_h_R}
v^{*_{h'_R}} = \sum_{i,a,j,b} R^{2\sigma(b-a)+2\chi_i(x)-2\chi_j(x)} v_{jb|ia}^{*_h}.
    \end{equation}
Let $I(d)$ be the set of indices $(i,a,j,b)$ such that $2\sigma(b-a)+2\chi_i (x)-2\chi_j (x)=d$. Define the bundle
    \begin{equation}
Q_{d}:= \bigoplus_{(i,a,j,b)\in I(d)} \text{Hom}(E^{i}_{a},E^{j}_{b}).
    \end{equation}
Sections of $Q_{d} \otimes \bigwedge^{p,q}T^{*}X$ are said to have \emph{degree $d$}. Note that if for a section $v$ of degree $d$, we have $v^{*_{h'_R}}=R^{d}v^{*_h}$. Next, we let 
    \begin{equation}
\mathcal{F}^{p,q}_{d} = \text{ the set of sections of } \bigoplus_{d'\geq d}Q_{d} \otimes \bigwedge^{p,q}T^{*}X.    
    \end{equation}
When we do not distinguish the type of forms, we will omit the superscripts. The following are a few simple but important observations.
\begin{lemma}\label{lemma_F_basic_properties}
    \begin{enumerate}[label=(\roman*)]
        \item
$\mathcal{F}^{p,q}_{d}\cdot \mathcal{F}^{p',q'}_{d'}\subset \mathcal{F}^{p+p',q+q'}_{d+d'}$.
        \item
If $v\in \mathcal{F}^{p,q}_{d}$ and $d>0$, then $\tr \ v=0$.
        \item
Elements of $\mathcal{H}_{0}$ have degree zero.
        \item
$(u-\phi_{j})^{1,0}\in \mathcal{F}_{2d_{j}-2\sigma}$ for $j=1,\cdots,k-1$.
        \item
$(u-\phi_{j})^{0,1}\in \mathcal{F}_{2d_{j}}$ for $j=1,\cdots,k-1$.
    \end{enumerate}
\end{lemma}
\begin{proof}
For $v\in \mathcal{F}^{p,q}_{d}$ and $v'\in \mathcal{F}^{p',q'}_{d'}$, we have
    \begin{equation*}
(vv')_{jb|ia} = \sum_{j',b'} v_{jb|j'b'}v'_{j'b'|ia}.
    \end{equation*}
(i) then follows from \eqref{equation_adjoint_h_R}. (ii) and (iii) are true because $v_{ia|ia}$ has degree zero for all $i$ and $a$. For (iv) and (v), we note that $(P_{S_{j}}(u)^{1,0})^{*_{h'_R}}=R^{2d_{j}-2\sigma}(P_{S_{j}}(u)^{1,0})^{*}$ and $(P_{S_{j}}(u)^{0,1})^{*_{h'_R}}=R^{2d_{j}}(P_{S_{j}}(u)^{1,0})^{*}$.
\end{proof}

At degree zero, we define $Q_{0}(h)$ be the subbundle of $h$-hermitian endomorphisms of $Q_0$. Observe that $Q_{0}(h)=Q_{0}(h'_R)$ for all $R>0$. We define
    \begin{equation}
Q:=Q_0(h)\oplus\bigoplus_{i>0}Q_{i}.
    \end{equation}
Let
    \begin{equation}
p:\sllie(E)\to Q        
    \end{equation}
be the bundle morhpism given by projection. We let $P(E,h)$ be the space of $L^{2}_{k}$-sections of $Q$ where $k$ is sufficiently large. For $f\in P(E,h)$, we write $f=\sum f_{d}$, where $f_d$ is the degree $d$ component of $f$. There is an isomorphism from $P(E,h)$ to the space of $L^{2}_{k}$-sections of $h'_R$-hermitian endomorphism, given by
    \begin{equation}\label{equation_f_R_def}
        f\mapsto f_R:= f_0+ \sum_{d>0}f_d + \sum_{d>0} R^{d} f_{d}^{*}.
    \end{equation}
Note that the space
$$H_0 = \{v\in \mathcal{H}^{0}: v^{*}=v\}$$
is contained in $P(E,h)$, by (iii) of Lemma \ref{lemma_F_basic_properties}. Lastly, we let $H_1$ to be the $L^2$-orthogonal complement of $H_0$ in $P(E,h)$.

Now, we return to the problem of the existence of the conformal limit of $(\dbar_u,\Phi_u)$. For $R>0$, define the connections
    \begin{equation}
D_{R}:= \dbar_{u}+\partial^{h'_{R}}_{u}+\hbar^{-1}\Phi_{u}+\hbar (R^{2\sigma} \Phi_u)^{*_{h'_{R}}},
    \end{equation}
which has the limit 
    \begin{equation}
D_0:=\lim_{R\to 0}D_R=\dbar_{E}+\partial^{h}_{E}+\Phi_{u}+\Phi^{*}
    \end{equation}
since by \eqref{equation_rescaled_R}, \eqref{equation_rescaled_R_2} and \eqref{equation_rescaled_R_3} we have $$\partial^{h'_R}_{u}=\partial^{h}-\beta^{*_{h'_R}}\xrightarrow{R\to 0} \partial^{h}_{E}$$
and $$(R^{2\sigma}\Phi_{u})^{*h'_R}=\Phi^{*}+R^{2\sigma}\varphi^{*_{h'_R}}\xrightarrow{R\to 0}\Phi^{*}.$$
More genearlly, for $f\in P(E,h)$, we have the expressions
    \begin{align}
D_{R}(f) &:= \dbar_{u}+\partial^{(e^{f_R})^*h'_{R}}_{u}+\hbar^{-1}\Phi_{u}+\hbar (R^{2\sigma} \Phi_u)^{*_{(e^{f_R})^{*}h'_{R}}} \nonumber\\
&= \dbar_{u}+\partial^{h'_R}_{u}+e^{-2f_R}(\partial^{h'_R}_{u}e^{2f_R})+\hbar^{-1}\Phi_{u}+\hbar e^{-2f_R}(R^{2\sigma}\Phi_u)^{*_{h'_R}}e^{2f_R}. \label{equation_general_D_R}     
    \end{align}
For $f\in P(E,h)$, consider the following operator as in \cite{DFK+21} and \cite{CW19}:
    \begin{equation}\label{equation_N_functional_original}
N(f,R):= F_{(\dbar_{u},(e^{f_R})^{*}h'_R)}+[\Phi_{u},(R^{2\sigma}\Phi_u)^{*_{(e^{f_R})^{*}h'_R}}].
    \end{equation}
Recall that $F_{(\dbar_{u},(e^{f_R})^{*}h'_R)}$ refers to the curvature of the Chern connection associated with the complex structure $\dbar_{u}$ and metric $(e^{f_R}))^{*}h'_R$, where $f_R$ is defined in \eqref{equation_f_R_def}. Similar to \cite{CW19}, one compute that
    \begin{align*}
[\Phi_{u},(R^{2\sigma}\Phi_u)^{*_{(e^{f_R})^{*}h'_R}}] &= [\Phi_{u},e^{-2f_R}(\Phi^{*}+R^{2\sigma}\varphi^{*_{h'_R}})e^{2f_R}]\\
\dbar_{u}(e^{-2f_R}(\partial^{h'_R}_ue^{2f_R})) &= [\beta,e^{-2f_R}(\partial^{h}_{E}e^{2f_R})-e^{2f_R}\beta^{*_{h'_R}}e^{2f_R}]+[\beta,\beta^{*_{h'_R}}]+\dbar_{E}(e^{-2f_R}\partial^{h'_R}_{u}e^{2f_R}) \\
F_{(\dbar_{u},h'_R)} &= F_{(\dbar_{E},h)}+\partial^{h}_{E}\beta-\dbar_{E}\beta^{*_{h'_R}}-[\beta,\beta^{*_{h'_R}}].
    \end{align*}
So, we have
    \begin{align}
F_{(\dbar_{u},(e^{f_R})^{*}h'_R)} &= \dbar_{u}(e^{-2f_R}((\partial^{h}_{E}-\beta^{*_{h'_R}})e^{2f_R}))+F_{(\dbar_{u},h'_R)} \nonumber\\
&= [\beta,e^{-2f_R}(\partial^{h}_{E}e^{2f_R})-e^{-2f_R}\beta^{*_{h'_R}}e^{2f_R}]+F_{(\dbar_{E},h)}+\partial^{h}_{E}\beta-\dbar_{E}\beta^{*_{h'_R}}+\dbar_{E}(e^{-2f_R}\partial^{h'_R}_{u}e^{2f_R}). \label{equation_N_functional_RHS}
    \end{align}
While $h'_R$ is only defined for $R>0$, the right side of \eqref{equation_general_D_R} and \eqref{equation_N_functional_RHS}, and thus $D_R(f)$ and $N(f,R)$, are defined for all $f$ and $R$.

Lastly, for any $f\in H_0$, we compute that $$N(f,0)=F_{(\dbar_{E},h)}+\partial^{h}\beta+[\Phi_u, \Phi^{*}]=F_{(\dbar_{E},h)}+[\Phi,\Phi^{*}]+D'u=0.$$

As explained earlier, one would have shown that the $\hbar$-conformal limit of $(\dbar_{u},\Phi_u)$ is $D_0$ if there exists, locally around $R=0$, a family of solutions $R\mapsto f(R)$ with $f(0)=0$ such that $N(f(R),R)\equiv 0$. Unfortunately, the presence of the non-trivial $H_0$ obstructs the construction of such $f(R)$ by means of a direct application of the implicit function theorem. This issue will be addressed with the help of Proposition \ref{proposition_iterative_algorithm}.

First, we note that similar to Lemma \ref{lemma_open_stratum_1st_term_redundant}, the last few terms at the right side of \eqref{equation_N_functional_RHS} are in some sense `orthogonal' to $H_0$.

\begin{lemma}
For all $v\in H_0$, $\int_{X}\tr \  p\bigg(i(F_{(\dbar_{E},h)}+\partial^{h}\beta-\dbar_{E}\beta^{*_{h'_R}}+\dbar_{E}(e^{-2f_R}\partial^{h'_R}_{u}e^{2f_R}))\bigg)v=0$.
\end{lemma}
\begin{proof}
Recall that $D''v=0$ and $D'v=0$. For any $A\in \Omega^{1}(\text{End}(E))$, we have $\int_{X}\tr \ p(\dbar_{E}A)v=\int_{X}\tr \ p(\dbar_{E}A_0)v$, where $A_0$ is the degree zero component of $A$. Thus, similar to the proof of Lemma \ref{lemma_open_stratum_1st_term_redundant}, we have $\int_{X}\tr \ p(\dbar_{E}A)v=0$. The case for $\int_{X}\tr \ p (\partial^{h}_{E}A)v$ is similar. Hence, the last three terms vanish. For the first term, recall that $iF_{(\dbar_{E},h)}=-i[\Phi,\Phi^{*}]$ and is $h$-hermitian. So $\int_{X}\tr \ p(-i[\Phi,\Phi^{*}])v =-i\int_{X}\tr \ [\Phi,\Phi^{*}]v=-i\int_{X}\tr \ \Phi^{*}[\Phi, v]=0$.
\end{proof}

Next, we will consider the functional $n: H_1\times H_0 \times \R \to H_0^{\vee}$, defined by
        \begin{align}
            n(f,g,R)(v) &:= \int_{X} \tr \ p\big(iN(f,g,R)\big)v \\
&= \int_{X} \tr \ \big\{p\bigg(i[\beta+\Phi_u, e^{-2f_R}(R^{2\sigma}\Phi_u-\beta)^{*_{h'_R}})e^{2f_R}+e^{-2f_R}(\partial^{h}_{E}e^{2f_R})]\bigg)v\big\} \label{equation_n_functional}\\
&= \text{Re}\int_{X}\tr \ i[\beta+\Phi_u, e^{-2f_R}(R^{2\sigma}\Phi_u-\beta)^{*_{h'_R}})e^{2f_R}+e^{-2f_R}(\partial^{h}_{E}e^{2f_R})]v. \label{equation_remove_p}
        \end{align}
Here, we have adopted the following conventions on notation. First, we write $N(f,g,R)=N(f+g,R)$ to emphasize that we treat $f$ and $g$ as two independent variables. However, to save spaces, we also denote $(f+g)_R=f_R+g_R=f_R+g$ simply by $f_R$.

\begin{lemma}\label{lemma_power_R}
For $n\geq 0$, we have
    \begin{enumerate}[label=(\roman*)]
        \item
$\dfrac{d^{n}}{dR^{n}}\bigg|_{R=0}2f_{R}\in \mathcal{F}^{0,0}_{-n}$
        \item
$\dfrac{d^{n}}{dR^{n}}\bigg|_{R=0}e^{2f_{R}}\in \mathcal{F}^{0,0}_{-n}$
        \item
$\dfrac{d^{n}}{dR^{n}}\bigg|_{R=0}\partial^{h}_{E}e^{2f_R}\in \mathcal{F}^{1,0}_{-n}$
        \item
$\dfrac{d^{n}}{dR^{n}}\bigg|_{R=0}\beta^{*_{h'_R}}\in \mathcal{F}^{1,0}_{-n}$
        \item
$\dfrac{d^{n}}{dR^{n}}\bigg|_{R=0} R^{2\sigma}\Phi_u^{*_{h'_R}}\in \mathcal{F}^{0,1}_{2\sigma-n}$.
    \end{enumerate}
\end{lemma}
\begin{proof}
(i), (iv), and (v) follow from $f_{d}^{*}\in \mathcal{F}^{0}_{-d}$. For (ii), first we have $$\dfrac{d^{n}}{dR^{n}}\bigg|_{R=0} \dfrac{(2f_R)^{j}}{j!}=\sum_{(j_k)\in \mathcal{P}(n)} \prod_{k}\dfrac{1}{j_k !}\dfrac{d^{j_k}}{dR^{j_k}}\bigg|_{R=0}2f_{R}.$$
Here, $\mathcal{P}(n)$ denotes the set of tuples $(j_k)_{k=1}^{j}$ of integers such that $j_k\geq 0$ and $\sum_{k=1}^{j}j_{k}=n$. Then, apply (i) and Lemma \ref{lemma_F_basic_properties}. Finally, for (iii), just note that $\partial^{h}_{E}$ preserves degrees.
\end{proof}

We finally arrive at the second important technical input for the main result of this section:

\begin{theorem}\label{theorem_n_functional_perp}
There exists a smooth map $g: U\times (-\epsilon,\epsilon)\to H_0$, where $U$ is an open neighborhood of $0$ in $H_1$ and $\epsilon>0$, such that $g(0,0)=0$ and $n(f,g(f,R),R)\equiv 0$.
\end{theorem}
\begin{proof}
Note that $[\Phi,v]=0$. For any $A\in \Omega^{1}(\text{End}(E))$, $\tr \  [\Phi,A]v=\tr \ A[\Phi,v]=0$. Hence, we can drop the Higgs field $\Phi$ in the term $[\beta+\Phi+\varphi, e^{-2f_R}(R^{2\sigma}\Phi_u-\beta)^{*_{h'_R}})e^{2f_R}+e^{-2f_R}(\partial^{h}e^{2f_R})].$

Recall that from Proposition \ref{proposition_iterative_algorithm}, we have determined a decreasing chain of tori $G=G_0\supset G_1\supset\cdots\supset G_k=\{1\}$, with Lie algebras denoted by $\text{Lie}(G_i)=\mathfrak{k}_{i}\otimes_{\R}\C$ for $i=0,\cdots, k$. Let $\mathfrak{k}_{i}'\otimes_{\R}\C$ be the Lie algebra of a complementary torus of $G_i$ in $G_{i-1}$. Then $H_0$ admits the following decomposition $$H_0=\bigoplus_{j=1}^{k}\sqrt{-1}\mathfrak{l}_{j}'.$$

For each summand $W$ of $H_0$, let $n_{W}$ be the restriction of $n$ to $W$ and $W^{\vee}$. For $1\leq j\leq k$ and $v\in i\mathfrak{k}'_j$, $[\phi_{j-1},v]=0$. Hence, by the same argument for $\Phi$ above, we have, for $g, v\in i\mathfrak{k}'_{j}$,
    \begin{equation}
        n_{i\mathfrak{k}_{j}'}(f,g,R)(v) = \text{Re}\int_{X} \tr \ i[u-\phi_{j-1}, e^{-2f_R}(R^{2\sigma}\Phi_u-\beta)^{*_{h'_R}})e^{2f_R}+e^{-2f_R}(\partial^{h}e^{2f_R})]v.
    \end{equation}
From Lemma \ref{lemma_F_basic_properties}, $(u-\phi_{j-1})^{1,0}\in \mathcal{F}_{2d_{j-1}-2\sigma}$ and $(u-\phi_{j-1})^{0,1}\in \mathcal{F}_{2d_{j-1}}$. By Lemma \ref{lemma_power_R}, we then have
    \begin{equation}
\dfrac{d^{n}}{dR^{n}}\bigg|_{R=0} [u-\phi_{j-1}, e^{-2f_R}(R^{2\sigma}\Phi_u-\beta)^{*_{h'_R}})e^{2f_R}+e^{-2f_R}(\partial^{h}e^{2f_R})]v\in \mathcal{F}^{1,1}_{2d_{j-1}-n}.
    \end{equation}
Using (ii) of Lemma \ref{lemma_F_basic_properties}, it follows that the functional
    \begin{equation}\label{equation_n_tilde}
\tilde{n}_{j}:= R^{-2d_{j-1}}n_{i\mathfrak{k}_{j}'}        
    \end{equation}
is well-defined everywhere. We compute
    \begin{align*}
n_{i\mathfrak{k}_{j}'}(0,g,R)(v) &= \text{Re}\int_{X} \tr \ i[u-\phi_{j-1}, e^{-2g}(R^{2\sigma}\Phi_u-\beta)^{*_{h'_R}})e^{2g}]v\\
&= \text{Re}\int_{X} \tr \ i[u-\phi_{j-1}, \Phi^{*}+e^{-2g}(R^{2\sigma}\varphi-\beta)^{*_{h'_R}})e^{2g}]v \\
&= \text{Re}\int_{X} \tr \ i[u-\phi_{j-1}, e^{-2g}(R^{2\sigma}\varphi-\beta)^{*_{h'_R}})e^{2g}]v.
    \end{align*}
Notice that the dependence on $R$ only occurs in the term $(R^{2\sigma}\varphi-\beta)^{*_{h'_R}}$. Therefore, from \eqref{equation_n_tilde}, we obtain
    \begin{equation}
n_{i\mathfrak{k}'_{j}}(0,g,R)(v) = \text{Re}\int_{X}\tr \ i[u-\phi_{j-1},e^{-2g}\big(R^{2\sigma}(u-\phi_{j-1})^{1,0}-(u-\phi_{j-1})^{0,1}\big)^{*_{h'_R}}e^{2g}]v,
    \end{equation}
from which we have
\begin{equation}\label{equation_n_tilde_vary_g}
\tilde{n}_{j}(0,g,0)(v) = \text{Re}\int_{X}\tr \ i[u-\phi_{j-1},e^{-2g}\bigg(\big(P_{S_{j-1}}(u))^{1,0}\big)^{*}-\big((P_{S_{j-1}}(u))^{0,1}\big)^{*}\bigg)e^{2g}]v.
    \end{equation}
In particular, putting $g=0$ and $R=0$, we have
    \begin{align*}
\tilde{n}_{j}(0,0,0)(v) &= \text{Re}\int_{X} \tr \ i[u-\phi_{j-1},\big((P_{S_{j-1}}(u))^{1,0}\big)^{*}-\big((P_{S_{j-1}}(u))^{0,1}\big)^{*}]v  \\
&= \text{Re}\int_{X} \tr \ i[P_{S_{j-1}}(u),\big((P_{S_{j-1}}(u))^{1,0}\big)^{*}-\big((P_{S_{j-1}}(u))^{0,1}\big)^{*}]v\\
&= \dfrac{1}{2}D_{0}\ell_{P_{S_{j-1}}(u)}(v),
    \end{align*}
where $\ell_{P_{S_{j-1}}(u)}: i\mathfrak{k}'_{j}\to \R$ is the Kempf-Ness functional (cf. \eqref{equation_KN_functional}) of $P_{S_{j-1}}(u)$ given by
$$\ell_{P_{S_{j-1}}(u)}(g)=||e^{g}P_{S_{j-1}}(u)e^{-g}||^{2}.$$
By the choice of the harmonic metric $h$ stipulated in Proposition \ref{proposition_iterative_algorithm}, and applying the Kempf-Ness theorem as in the proof of Lemma \ref{lemma_local_minima}, we have \begin{inparaenum}[(i)]
    \item $\tilde{n}_{j}(0,0,0)=0$, and
    \item the differential at $0$ along the $i\mathfrak{k}_{j}'$-direction is non-degenerate.
\end{inparaenum} Hence, by the implicit function theorem, there exists an open neighborhood $U_j$ of $0$ in $H_1\oplus \bigoplus_{j'\neq j}i\mathfrak{k}'_{j'}\oplus \R$, such that there is a unique map
    \begin{equation}
\tau_{j}:U_j\to i\mathfrak{k}_{j}',
    \end{equation}
that satisfies $$\tau_{j}(0,0,0)=0, \text{ and }$$
$$\tilde{n}_{j}(f,\hat{g}_{j}+\tau_j (f,\hat{g}_j,R),R)\equiv 0,$$
where we have made the following convention on the notation used: for $g\in H_0$, we write $g=\sum_{j'=1}^{k}g_{j'}$ where $g_{j'}\in i\mathfrak{k}_{j'}'$ for each $j'=1,\cdots, k$, and $\hat{g}_{j}:=g-g_{j}$. Putting the factor $R^{2d_{j-1}}$ back, we see that the same holds for $n_{i\mathfrak{k}_{j}'}$.

If $k=1$, taking $g(f,R)$ to be $\tau_1$ will finish the proof. In what follows, we assume that $k\geq 2$. Recall that $P_{S_{j-1}}(u)$ is fixed by $G_{j}$. For each $j< k$ and $v\in i\mathfrak{k}'_{j}$, from \eqref{equation_n_tilde_vary_g} it follows that for all $v$,
$$\tilde{n}_{j}(0,\sum_{j'>j}g_{j'},0)(v)\equiv 0.$$
By uniqueness of $\tau_{j}$, we have
$$D_{(0,0,0)}\tau_{j}(\sum_{j'>j}w_{j'}) \equiv 0$$
for $w_{j'}\in i\mathfrak{k}_{j'}'$, $j'=j+1,\cdots, k$.

Finally, consider the map $$F=(\tau_1,\cdots, \tau_k)-\text{pr}_{H_0}: \bigcap_{j=1}^{k}(U_{j}\times i\mathfrak{k}'_{j})\to H_0,$$
where $\text{pr}(f,g,R):= g$ is the projection map. From the above, it follows that $D_{(0,0,0)}F\big|_{H_0}$ is block lower-triangular with $-\text{id}$ on the diagonal. Restricting to a product neighborhood and applying the implicit function theorem, there exists a map $$g: U\times (-\epsilon,\epsilon)\to H_0 $$ from an open neighborhood of $0$ of $H_1\times \R$ to $H_0$ such that $g(0,0)=0$ and $F(f,g(f,R),R)\equiv 0$. In other words, $\tau_j(f,\hat{g}_j(f,R),R)\equiv g_{j}(f,R)$ for each $j$. For each $j$ and $v\in i\mathfrak{k}'_{j}$,
$$n(f,g(f,R),R)(v)=n_{i\mathfrak{k}'_{j}}(f,\hat{g}_{j}+\tau_{j},R)(v)=0.$$
Therefore, $n(f,g(f,R),R)\equiv 0$.
\end{proof}

\begin{corollary}\label{corollary_conformal_abelian}
The $\hbar$-conformal limit of $(\dbar_{u},\Phi_u)\in S_{+}$ is $D_0=\dbar_{E}+\partial^{h}_{E}+\hbar^{-1}\Phi_{u}+\hbar \Phi^{*}$.
\end{corollary}
\begin{proof}
Consider $$N'(f,R):=p\big(iN(f,g(R),R)\big).$$
By Theorem \ref{theorem_n_functional_perp}, $\text{im}N'\perp H_0$. Differentiating at $(0,0)$ along the $H_1$-direction, we obtain
$$D_{(0,0)}N'(f)=2p\bigg(iD_{u}''D'(f+D_{(0,0)}g(f))\bigg)=2p\bigg(iD_{u}''D'f\bigg)$$
since $D_{(0,0)}g(f)\in H_0\subset \ker D'$.

We claim that $\ker D_{(0,0)}N'=\{0\}$. Suppose $f\in \ker D_{(0,0)}N'$. Then $$\big(p(iD_{u}''D' f)\big)_0=p(iD''D'f_0)=0.$$
Now, $$(D''D'f_0)^{*}=(\dbar_{E}\partial^{h}f_0)^{*}+([\Phi,[\Phi^{*},f_0]])^{*}=\partial^{h}\dbar_{E}f_0^{*}+[\Phi^{*},[\Phi,f_0^{*}]]=D'D''f_0,\text{ and }$$
$$0=(D''+D')^{2}f_0=D''D'f_0+D'D''f_0.$$
It follows that $0=p(iD''D'f_0)=iD''D'f_0$. But this implies $f_0\in H_0$, so $f_0=0$. Then, using (iii) in Lemma \ref{lemma_F_basic_properties}, one can argue inductively similar to \cite{CW19}, showing that $f_d=0$ if $f_{d'}=0$ for all $d'<d$. Hence, $f=0$ and the claim follows.

The index of differential operator $p\circ iD_{u}''D'$ is the same as that of $p\circ iD''D'=iD''D'|_{P(E,h)}$, the latter being zero. Hence, we can apply the implicit function theorem, obtaining the solution $f(R)$ for small $R$ with $f(0)=0$, so that $N'(f(R),R))\equiv 0$.

To see that $N(f(R),R)\equiv 0$, analogous to Lemma \ref{lemma_open_stratum_small_projection}, it suffices to prove the following

\begin{lemma}
There exists $\delta>0$ such that for continuous $f$ with sup-norm $||f||_{h}<\delta$, the projection map $p$ gives an isomorphism between continuous sections of $(e^{f_R})^{*}h'_R$-hermitian endomorphisms and continuous sections of $Q$.
\end{lemma}
    \begin{proof}
Let $x\in X$. At a fibre, it suffices to show that the restriction of $p$ to $i\mathfrak{u}(E_x,(e^{f_R})^{*}h'_R)$, the space of $(e^{f_R})^{*}h'_R$-hermitian endomorphisms of $E_x$, is injective.

Let $$Q'=iQ_0(h)\oplus \bigoplus_{d<0}Q_d$$
be the complementary bundle of $Q$. In other words, $Q'$ is the kernel of $p$. For $g\in Q_{x}'$, define $$g(R):= \sum_{d\leq 0}R^{-d} g_d.$$
Consider the map $\beta: Q_x\times Q_{x}'\times \R\to \text{End}(E)_x$ defined by $$\beta(f,g,R)=g(R)-e^{-2f_R}g^{*} e^{2f_R}.$$
Note that for fixed $f$ and $R$, $\beta$ is linear in $g$. For $f=0$ and $R=0$, we have $\beta(0,g,0)=g_0-g^{*}$, which vanishes if and only if $g=0$. Hence, $\beta(0,\cdot,0)$ is injective. Then, for small $R$ and $f$ with small $h$-norm, $\beta(f,\cdot,R)$ is injective. For such $f$ and positive $R$, $\beta(f,g,R)=g(R)-g(R)^{*_{(e^{f_R})^{*}h'_R}}$, so $g(R)$ is $(e^{f_R})^{*}h'_R$-hermitian if and only if $g=0$. Finally, suppose $n\in i\mathfrak{u}(E_x,(e^{f_R})^{*}h'_R)$ and $p(n)=0$. Then $n\in Q_{x}'$. Since $n=n'(R)$ for some $n'\in Q_{x}'$, $\beta(f,n',R)=0$. Thus, $n'=n=0$.

Choosing local unitary frames and using compactness of $X$, the result then follows.
    \end{proof}

Since $k$ is sufficiently large, $N(f(R),R)\in C^{0}(\Omega^{1,1}(\text{End}(E))$. Therefore, we can conculde the proof by the above lemma.
\end{proof}